\documentclass{article}

\usepackage{tikz}
\usetikzlibrary{positioning}

\usepackage{hyperref}
\usepackage{amsmath}
\usepackage{amsthm}
\usepackage{amssymb}
\usepackage{amsfonts}
\newcommand{\field}[1]{\mathbb{#1}}
\newtheorem{thm}{Theorem}[section]
\newtheorem{cor}[thm]{Corollary}

\newtheorem{lem}[thm]{Lemma}
\newtheorem{prop}[thm]{Proposition}
\newtheorem{remk}[thm]{Remark}

\numberwithin{equation}{section}
\usepackage[top=30mm, bottom=30mm, left=30mm, right=30mm]{geometry}
\allowdisplaybreaks
\usepackage{graphicx}

\begin{document}
\title{On two-signed solutions to a second order semi-linear parabolic partial differential equation with non-Lipschitz nonlinearity}
\author{V. Clark and J. C. Meyer}

\maketitle

\begin{abstract}
In this paper, we establish the existence of a 1-parameter family of spatially inhomogeneous radially symmetric classical self-similar solutions to a Cauchy problem for a semi-linear parabolic PDE with non-Lipschitz nonlinearity and trivial initial data. 
Specifically we establish well-posedness for an associated initial value problem for a singular two-dimensional non-autonomous dynamical system with non-Lipschitz nonlinearity.
Additionally, we establish that solutions to the initial value problem converge algebraically to the origin and oscillate as $\eta\to \infty$.
\end{abstract}

\section{Introduction} \label{intro}

In this paper we consider $u : \bar{D}_T \rightarrow \mathbb{R}$ such that $u=u(x,t)$ is continuous and bounded on $ \bar{D}_T:=\mathbb{R} ^{n} \times [0, T] $ and, for fixed $n \in \mathbb{N}$, $u_{t}$, $u_{x_{i}}$ and $u_{x_{i}x_{j}}$ exist and are continuous on $D_T:= \mathbb{R}^{n} \times (0, T]$ for each $1 \leq i,j \leq n$.
Moreover, we suppose that $u$ is a solution to the following Cauchy problem for the second order semi-linear parabolic partial differential equation with non-Lipschitz (H\"older continuous) nonlinearity, given by 
\begin{align}
\label{i1} & u_{t} - \Delta u = u|u|^{p-1} \quad \mathrm{on} \quad D_{T}, \\
\label{i2} & u = 0 \quad \text{ on } \quad \partial D_T , \\
\label{i3} & u \in C^{2,1}( D_{T}) \cap C(\bar{D}_{T}) \cap L^{\infty} (\bar{D}_{T} ) ,
\end{align}
with $T>0$, $0<p<1$ and $\partial D_T:= \mathbb{R}^{n} \times \{ 0 \}$.
Here $C^{2,1}(X)$ denotes the set of functions that are defined on $X$ which are continuously differentiable twice with respect to the spatial variables $x$, and once with respect to the time variable $t$; 
$C(X)$ denotes the set of functions that are defined and continuous on $X$; 
and $L^{\infty} (X)$ denotes the set of functions with bounded essential supremum and infimum. 
We refer to the Cauchy problem in \eqref{i1}-\eqref{i3} as [CP] and $u: \bar{D}_T \rightarrow \mathbb{R}$ satisfying \eqref{i1}-\eqref{i3} as a solution to [CP].
In addition, throughout the paper we denote $(x,t) \in \bar{D}_{T}$ as $(x_{1}, x_{2},\dots, x_{n}, t) $, for $x \in \mathbb{R}^n$, $t\in [0,T]$. 

The existence of spatially inhomogeneous classical self-similar solutions to [CP] with $n=1$ has been considered in detail in \cite{meyer}.
In the paper, via consideration of a self-similar solution structure, a two-dimensional non autonomous dynamical system with non-Lipschitz nonlinearity was analysed and the existence of a two parameter family of homoclinic connections on the equilibrium point $(0,0)$ of the dynamical system, as well as decay bounds and estimates on these connections, were established. 
Herein, we consider an analogously derived dynamical system in $n$-spatial dimensions, for $n\in\mathbb{N}$, and establish the existence of spatially inhomogeneous solutions to \eqref{i1}-\eqref{i3}.
Moreover, we establish a full well-posedness result for the initial value problem for the dynamical system.
Furthermore, we prove that solutions oscillate as $\eta \to \infty$, which gives additional structural information about the aforementioned solutions in \cite{meyer}.
Curiously, oscillation theory of Sturmian type (see, for example \cite{kurt} or \cite{prot}), when combined with algebraic decay bounds on solutions to the initial value problem for the dynamical system as $\eta \to \infty$, obtained here via an adaptation of a technical argument in \cite{weis}, appear to be insufficient to establish oscillation of solutions as $\eta \to \infty$.
Hence, we adopt a novel alternative approach which relies on properties of non-negative solutions to [CP] established in \cite{ag}, to establish that solutions to the initial value problem for the dynamical system oscillate as $\eta \to \infty$.  
  
Qualitative properties of non-negative classical bounded solutions to boundary value problems for \eqref{i1}, have been considered in \cite{ag}, \cite{king}, \cite{cabe}, \cite{Meyer15}, \cite{mey2}, \cite{meyer} and \cite{need} with $0<p<1$ and non-negative initial data, and until [CP] in \cite{meyer} with $n=1$, two-signed solutions were not considered. 
We highlight here that the spatially inhomogeneous solutions constructed in this paper are two signed on $\bar{D}_T$ because any non-negative classical bounded solution to [CP] must be spatially homogeneous \cite[Corollary 2.6]{ag}. 
Following the investigations in \cite{Meyer1} and \cite{mey} it should be noted that local results which guarantee spatial homogeneity of solutions to semi-linear parabolic Cauchy problems, with homogeneous initial data, depend critically on uniqueness results (which do not apply to [CP] with $0<p<1$). 

Non-negative classical bounded solutions to boundary value problems for \eqref{i1} have been extensively investigated with $p \geq 1$, see for example \cite{caz}, \cite{doh}, \cite{esc}, \cite{fuji}, \cite{weis}, \cite{hi}, \cite{naito}, \cite{cham},  \cite{pol}, \cite{shi}, \cite{fb} and \cite{fbw}. 
These consider: conditions required for global solutions to exist and qualitative properties of solutions, such as asymptotic structure as $t \to \infty$ or $|x| \to \infty$; existence or non-existence of one signed solutions; and critical exponents which characterise solution structure to associated boundary value problems. 
Two review articles \cite{deng} and \cite{lev} consider a broad overview of this field of research. 
Additionally solutions to \eqref{i1}-\eqref{i3} with two signed initial data and with $p>1$ have been investigated in \cite{miz} and \cite{miz2}. 
The associated Dirichlet boundary value problem for \eqref{i1} on bounded spatial domains have been considered in \cite{ball} and \cite{cazy}. 
 
The remainder of the paper is structured as follows: in \S \ref{pf} we establish {\emph{a priori}} bounds on solutions to [CP], and subsequently, we formulate a radial self-similar solution structure of [CP] which gives an associated initial value problem for a two-dimensional singular non-autonomous dynamical system with non-Lipschitz nonlinearity studied in the remainder of the paper and referred to throughout as (P) (see Lemma \ref{l23} for details). 
In \S \ref{gwp}, we proceed to show that there exists a local solution to (P), by using a suitable contraction mapping, which can be extended to a global solution via {\emph{a priori}} bounds and multiple applications of the Cauchy-Peano Local Existence Theorem. 
We subsequently establish that there exists a 1-parameter family of solutions to (P) which converges to $(0,0)$ as $\eta \to \infty$. 
We complete the section by establishing well-posedness of (P), via uniqueness and continuous dependence results for $\eta \in [0, \infty)$ and initial data in $(0, (1-p)^{\frac{1}{1-p}}) \times \{0 \}$ (see Theorem \ref{t314} for details).
In \S \ref{qp}, we establish algebraic decay bounds for solutions to (P) as $\eta \to \infty$.
Furthermore we demonstrate that solutions to (P) oscillate as $\eta \to \infty$.
In \S \ref{conc} we summarise the main result established in the paper and explain how related results in \cite{meyer} can be improved.
We also highlight potential extensions to results in \S \ref{gwp}-\S \ref{qp} and related queries that have arisen from the study.

\section{Self-similar solution structure to [CP]} \label{pf}

In this section, we establish {\emph{a priori}} bounds for solutions to [CP]. 
Consequently, we consider a radial self-similar solution structure of solutions to [CP] which yields (P).
To begin, we have,

\begin{lem} \label{l21}
Let $u:\bar{D}_{T} \to\field{R}$ be a solution to [CP]. 
Then,
\begin{equation*} \label{l21a} |u(x,t)| \leq ((1-p)t)^{1/(1-p)} \ \ \  \forall (x,t) \in \bar{D}_{T}. \end{equation*}
\end{lem}

\begin{proof}
Since [CP] has spatially homogeneous initial data, the maximal solution $u^+:\bar{D}_T\to\mathbb{R}$ and minimal solution $u^-:\bar{D}_T\to\mathbb{R}$ to [CP] are spatially homogeneous for $t\in [0,T]$ (see, for example, \cite[Proposition 8.31]{mey} for $n=1$).
We note that $u^\pm :\bar{D}_T\to\field{R}$ given by 
\[ u^\pm(x,t)= \pm ((1-p)t)^{1/(1-p)} \ \ \ \forall (x,t) \in \bar{D}_{T} \]
are the maximal and minimal solutions to [CP], and hence any solution $u:\bar{D}_T\to\mathbb{R}$ to [CP] satisfies $u^- \leq u \leq u^+$ on $\bar{D}_T$, as required. 
\end{proof}

We now determine an initial value problem [IVP] associated with a self-similar solution structure to [CP].
Consider a solution $u: \bar{D}_{T} \to \field{R}$ to [CP] on $\bar{D}_{T}$ of the following form: there exists $w: [0, \infty) \to \field{R}$ such that
\begin{equation} \label{s2a} 
u(x,t)= \begin{cases} w\left(\frac{|x|}{t^{1/2}}\right)t^{\frac{1}{(1-p)}}, & (x,t) \in D_{T} \\ 0, & (x,t) \in \partial D_T. \end{cases} 
\end{equation}
 
We introduce $H:\field{R} \to \field{R}$ given by
\begin{equation} \label{s2l} 
H(w)= \begin{cases} \frac{1}{(1-p)}w -w|w|^{p-1} , & x\in\mathbb{R}\setminus \{ 0 \} \\ 0, & x=0  \end{cases}
\end{equation}
and observe that $H\in C(\field{R}) \cap C^1(\field{R} \setminus \{0\})$. 
We also denote
\begin{equation} \label{s2m} 
M_{H} = \sup_{[0, (1-p)^{1/(1-p)}]} |H| >0\ \text{ and }\ m_{H} = \inf_{[0, (1-p)^{1/(1-p)}]} H <0 . 
\end{equation}
We note here that via Lemma \ref{l21}, any solution to [CP] of the form \eqref{s2a} satisfies
\begin{equation} \label{s2j} 
||w||_\infty \leq (1-p)^{1/(1-p)} .
\end{equation}
We also note that if there exists a solution to [CP] of the form \eqref{s2a}, then $-u$ is also a solution to [CP].

It follows from \eqref{s2j} that $u: \bar{D}_{T} \to \field{R}$ given by \eqref{s2a} is a solution to [CP] (up to symmetry) if and only if there exists a constant $\alpha \in [0,(1-p)^{1/(1-p)}]$ such that $w : [0, \infty) \to \field{R}$ satisfies the following [IVP] for the second order singular non-autonomous ordinary differential equation with non-Lipschitz nonlinearity given as,
\begin{align} 
\label{s2b} & w'' +  \left(\frac{(n-1)}{\eta} + \frac{\eta}{2}\right) w' - H(w) = 0 \ \ \  \forall \eta \in (0,\infty), \\
\label{s2c} & w(0)= \alpha, \quad w'(0)=0, \quad \alpha \in [0,(1-p)^{1/(1-p)}], \\
\label{s2d} & w \in C^{2}([0,\infty)) \cap L^{\infty}([0,\infty)). 
\end{align}
Observe that the condition on $w'(0)$ ensures that $u$ given by \eqref{s2a}, has continuous first spatial derivatives on $\bar{D}_T$ for $t\in [0,T]$.
Moreover, from \eqref{s2b}-\eqref{s2d} it follows that
\begin{equation} \label{s2c'} 
w''(0) = \frac{H(\alpha)}{n} , 
\end{equation}
and hence, $u$ satisfies \eqref{i1} on $D_T$. 
Note that the [IVP] given by \eqref{s2b}-\eqref{s2d} is equivalent to the [IVP] for the singular two-dimensional non-autonomous dynamical system with non-Lipschitz right hand side, given by;
\begin{align}
\label{s2f} & (w)' = w' , \\
\label{s2g} & (w')' = H(w) - \left(\frac{(n-1)}{\eta} + \frac{\eta}{2}\right) w' \ \ \ \forall \eta \in (0,\infty), \\
\label{s2h} & (w(0), w'(0))=\left(\alpha, 0 \right), \quad \alpha \in [0,(1-p)^{1/(1-p)}], \\
\label{s2i} & (w,w') \in C^{1}([0,\infty)) \cap L^{\infty}([0,\infty)). 
\end{align}

Due to the singular term in \eqref{s2b} at $\eta = 0$, we give a specific argument to establish that there exists a solution to \eqref{s2b}-\eqref{s2d}. 
It is also convenient to express the [IVP] given by \eqref{s2b}-\eqref{s2d} as an integral equation, and hence, we have, 

\begin{lem} \label{l23}
The following statements are equivalent
\begin{enumerate}
\item[(a)] $w:[0,\infty) \to \field{R}$ is a solution to the [IVP] given by \eqref{s2b}-\eqref{s2d}.
\item[(b)] $(w,w'):[0,\infty) \to \field{R}^2$ is a solution to the [IVP] given by \eqref{s2f}-\eqref{s2i}.
\item[(c)] $w:[0,\infty) \to \field{R}$ satisfies
\begin{align} 
\label{l23b} & w(\eta) = \alpha + \int_{0}^{\eta} \frac{1}{t^{n-1}e^{\frac{t^2}{4}}} \int_{0}^{t} H(w(s))s^{n-1}e^{\frac{s^2}{4}} dsdt \ \ \ \forall \eta \in [0,\infty) , \\ 
\label{l23b'} & \alpha \in [0,(1-p)^{1/(1-p)}], \\
\label{l23a} & w \in C([0,\infty)) \cap L^{\infty}([0,\infty)) . 
\end{align}
\end{enumerate}
\end{lem}

\begin{proof}
It follows immediately that (a) and (b) are equivalent. 
Now, suppose that $w$ satisfies (a).
By multiplying \eqref{s2b} by $e^{\frac{\eta^2}{4}} \eta^{n-1}$ and integrating twice, it follows that $w$ satisfies \eqref{l23b}, and since (a) implies \eqref{l23a}, then $w$ satisfies $(c)$. 
Now suppose $w$ satisfies (c). 
From \eqref{l23b} and \eqref{l23a}, it follows that $w \in C^{1}([0,\infty)) \cap L^{\infty}([0,\infty))$ with 
\begin{equation} \label{l23c} 
w'(\eta) = \frac{1}{\eta^{n-1}e^{\frac{\eta^2}{4}}} \int_{0}^{\eta} H(w(s))s^{n-1}e^{\frac{s^2}{4}}ds \ \ \ \forall \eta \in (0,\infty), 
\end{equation}
\begin{equation*} \label{l23d} 
w(0)=\alpha, \ \ \ w'(0)=0. 
\end{equation*}
Additionally, from \eqref{l23c} it follows that $w \in C^2((0,\infty))$ with
\begin{equation} \label{l23e} 
(\eta^{n-1}e^{\frac{\eta^2}{4}}w'(\eta))'= H(w(\eta))\eta^{n-1}e^{\frac{\eta^2}{4}} \ \ \  \forall \eta \in (0,\infty), 
\end{equation}
and that $w''$ is continuous at $\eta =0$, with
\begin{equation*} \label{l23f} 
w''(0)= \lim_{\eta\to 0^{+}} \frac{w'(\eta)}{\eta} = \frac{H(\alpha)}{n}. 
\end{equation*}
In addition, $w''(\eta)$ satisfies \eqref{l23e} for all $\eta \in (0,\infty)$, so that 
\begin{equation*} \label{l23g} 
w''(\eta) =-\left(\frac{\eta}{2}+\frac{(n-1)}{\eta}\right)w'(\eta) + H(w(\eta)) \ \ \ \forall \eta \in (0,\infty). 
\end{equation*}
Hence $w \in C^2 ([0,\infty)) \cap L^{\infty}([0,\infty))$ and satisfies \eqref{s2b}. 
Thus it follows that $w$ satisfies (a).
Hence (a), (b) and (c) are equivalent, as required. 
\end{proof}

We refer to the equivalent [IVP] given by (a), (b) and (c) in Lemma \ref{l23} as (P).

\section{Well-posedness of (P)} \label{gwp}

In this section, we establish that (P) is well posed in the sense of Hadamard, for initial data $w(0)\in [0, (1-p)^{1/(1-p)})$.

\subsection{Existence}

We first establish a local existence result for solutions to (P) on $[0,\epsilon ]$ via a contraction mapping, and then extend this to an existence result for (P), via multiple applications of the Cauchy-Peano Theorem. 

\begin{thm} \label{t31}
(P) has a unique local solution on $[0,\epsilon]$ with
\begin{equation} \label{t31a} 
\epsilon= \min \left \{ \left(\frac{\alpha}{\sup\limits_{\frac{1}{2} \alpha \leq w \leq \frac{3}{2} \alpha } |H(w)|}\right)^{\frac{1}{2}}, \left(\frac{1}{(1-p)} + p\left(\frac{\alpha}{2}\right)^{p-1} \right)^{-\frac{1}{2}} \right \}, 
\end{equation}
\end{thm}

\begin{proof}
Consider the Banach space $X = ((C[0, \epsilon]), ||\cdot||_{\infty})$ and the closed subset of $X$, given by
\begin{equation} \label{t31d} 
D = \left\{u \in C([0,\epsilon]): \frac{\alpha }{2} \leq u(x) \leq \frac{3\alpha}{2} \right\}. 
\end{equation}
Moreover, we define the operator $T:C([0,\epsilon]) \to C([0, \epsilon])$, given by,
\begin{equation*} \label{t31c} T[w](\eta)= \alpha + \int_{0}^{\eta} \frac{1}{t^{n-1} e^{\frac{t^2}{4}}} \int_{0}^{t} H(w(s))s^{n-1}e^{\frac{s^2}{4}} dsdt \ \ \  \forall w\in C([0,\epsilon ]), \eta \in [0, \epsilon]. \end{equation*}
For $w_{1} \in D$, set $I\in C([0,\epsilon ])$, to be
\begin{equation} \label{t31e} I(\eta ) = \int_{0}^{\eta} \frac{1}{t^{n-1} e^{\frac{t^2}{4}}} \int_{0}^{t} H(w_{1}(s))s^{n-1}e^{\frac{s^2}{4}} dsdt \ \ \  \forall \eta \in [0,\epsilon]. \end{equation}
Observe that
\begin{align}
\nonumber |I(\eta)| &\leq \int_{0}^{\eta} \frac{1}{t^{n-1} e^{\frac{t^2}{4}}} \int_{0}^{t} |H(w_{1}(s))|s^{n-1}e^{\frac{s^2}{4}} dsdt \\
\nonumber &\leq \sup_{\frac{1}{2} \alpha \leq w \leq \frac{3}{2} \alpha} |H(w)| \int_{0}^{\eta} \int_{0}^{t} \frac{s^{n-1}e^{\frac{s^{2}}{4}}}{t^{n-1}e^{\frac{t^{2}}{4}}} dsdt \\
\label{t31f} &\leq \sup_{\frac{1}{2} \alpha \leq w \leq \frac{3}{2} \alpha } |H(w)| \frac{\eta^{2}}{2}
\end{align}
for all $\eta \in [0,\epsilon]$. 
It follows that $|I(\eta)| \leq \frac{\alpha}{2}$ provided that
\begin{equation} \label{t31g} \epsilon < \left(\frac{\alpha}{\sup\limits_{\frac{1}{2} \alpha \leq w \leq \frac{3}{2} \alpha } |H(w)|}\right)^{\frac{1}{2}}. \end{equation}
Since $\epsilon$ given by \eqref{t31a} satisfies \eqref{t31g}, it follows from \eqref{t31f} and \eqref{t31e} that $T[w_{1}] \in D$ for all $w_{1} \in D$ and hence $T[D] \subset D$. 
Now, consider
\begin{equation} \label{t31h} || T[w_{1}]-T[w_{2}] ||_{\infty} \leq \int_{0}^{\epsilon} \frac{1}{t^{n-1} e^{\frac{t^2}{4}}} \int_{0}^{t} s^{n-1} e^{\frac{s^2}{4}} || H(w_{1}(\cdot))- H(w_{2}(\cdot))||_{\infty} dsdt \ \ \  \forall w_{1}, w_{2} \in D. \end{equation}
Observe that $H\in C^1(\mathbb{R}\setminus \{ 0 \})$, given by \eqref{s2l} satisfies
\begin{equation} \label{t31j} |H'(w)| \leq \frac{1}{(1-p)} + p \left(\frac{\alpha}{2} \right)^{p-1} =: C_{\alpha} \ \ \ \forall x \in \left[\frac{\alpha}{2}, \frac{3\alpha}{2} \right]. \end{equation}
Furthermore, via \eqref{t31h}, \eqref{t31j} and \eqref{t31a} it follows that
\begin{align}
\nonumber || T[w_{1}] - T[w_{2}] ||_{\infty} & \leq C_{\alpha} || w_{1}-w_{2} ||_{\infty} \int_{0}^{\epsilon} \frac{1}{t^{n-1} e^{\frac{t^2}{4}}} \int_{0}^{t} s^{n-1}e^{\frac{s^2}{4}} dsdt \\
\nonumber & \leq \frac{C_\alpha \epsilon^2}{2}  || w_1-w_2 ||_\infty \\
\label{3.11} & \leq \frac{1}{2} ||w_1 - w_2||_\infty \ \ \ \forall w_1,w_2\in D.
\end{align}
We conclude from \eqref{3.11} that $T$ is a contraction mapping on $D$, and via the contraction mapping principle, there exists a unique fixed point $w^{*} \in D$ of $T$. 
It follows from \eqref{t31d}, \eqref{t31e} and Lemma \ref{l23} that $w^*$ is the unique solution to (P) restricted to $[0,\epsilon]$, as required.
\end{proof}

We now illustrate that the local solution to (P) on $[0,\epsilon ]$ can be extended to a solution to (P) on $[0,\infty )$.
First introduce $Q:\field{R}^{2} \times (0, \infty) \to \field{R}^{2}$, given by 
\begin{equation}  \label{s3a} Q(w,w',\eta) = \left(w', H(w) -\left(\frac{(n-1)}{\eta} + \frac{\eta}{2}\right)w' \right) \ \ \  \forall (w,w',\eta)\in \mathbb{R}^{2} \times (0,\infty), \end{equation}
and note that $Q \in C(\field{R}^{2} \times (0, \infty))$, but also that $Q$ is not locally Lipschitz continuous on $\field{R}^{2} \times (0, \infty)$ ($Q$ is locally Lipschitz continuous on $\field{R}^{2} \times (0, \infty) \setminus \mathcal{N}$, with $\mathcal{N}$ any neighbourhood of the plane $w=0$).
We also introduce the function $V:\field{R}^{2} \to \field{R}$ defined by, 
\begin{equation} \label{s3b} V(w,w') = \frac{1}{2}(w')^{2} - \frac{1}{2(1-p)}w^{2} + \frac{1}{(1+p)}|w|^{1+p} \ \ \  \forall (w,w') \in \field{R}^{2}. \end{equation}
We observe immediately that $V \in C^{1,1}(\field{R}^2)$ with
\begin{equation} \label{s3d} \nabla V(w,w')= \left( \frac{-w}{(1-p)}+w|w|^{p-1}, w' \right) \quad \forall (w,w') \in \field{R}^{2}. \end{equation}

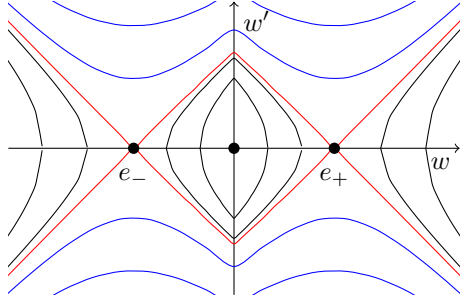
\begin{figure}[h!] \label{fig:1}
\centering
\begin{tikzpicture}[scale=1.5]
	
		\clip (-2,-1.3) rectangle + (4,2.6);

      \draw[->] (-2,0) -- (2,0) node[below left] {$\ w$};
      \draw[->] (0,-1.3) -- (0,1.3) node[below right] {$w'$};

      \draw[domain=-3:3,smooth,variable=\x,blue] plot ({\x},{sqrt(2*(0.95 + (1/(2*(1-0.1))*(\x*\x) - (1/(1+(0.1)))*((\x*\x)^(0.5))^(1+0.1)))});
      \draw[domain=-3:3,smooth,variable=\x,blue] plot ({\x},{-sqrt(2*(0.95 + (1/(2*(1-0.1))*(\x*\x) - (1/(1+(0.1)))*((\x*\x)^(0.5))^(1+0.1)))});
      \draw[domain=-3:3,smooth,variable=\x,blue] plot ({\x},{-sqrt(2*(0.55 + (1/(2*(1-0.1))*(\x*\x) - (1/(1+(0.1)))*((\x*\x)^(0.5))^(1+0.1)))});
      \draw[domain=-3:3,smooth,variable=\x,blue] plot ({\x},{sqrt(2*(0.55 + (1/(2*(1-0.1))*(\x*\x) - (1/(1+(0.1)))*((\x*\x)^(0.5))^(1+0.1)))});
      \draw[domain=-3:0.1,smooth,variable=\x,blue] plot ({\x},{sqrt(2*(1.7 + (1/(2*(1-0.1))*(\x*\x) - (1/(1+(0.1)))*((\x*\x)^(0.5))^(1+0.1)))});

      \draw[domain=-{(1-0.1)^(1/(1-0.1))+0.03}:{(1-0.1)^(1/(1-0.1))-0.03},smooth,variable=\x,red] plot ({\x},{-sqrt(2*(((1-0.1)^(2/(1-0.1)))/(2*(1+0.1)) + (1/(2*(1-0.1))*(\x*\x) - (1/(1+(0.1)))*((\x*\x)^(0.5))^(1+0.1)))}); 
      \draw[domain=-{(1-0.1)^(1/(1-0.1))+0.03}:{(1-0.1)^(1/(1-0.1))-0.03},smooth,variable=\x,red] plot ({\x},{sqrt(2*(((1-0.1)^(2/(1-0.1)))/(2*(1+0.1)) + (1/(2*(1-0.1))*(\x*\x) - (1/(1+(0.1)))*((\x*\x)^(0.5))^(1+0.1)))});

      \draw[domain={(1-0.1)^(1/(1-0.1))+0.03:3},smooth,variable=\x,red] plot ({\x},{sqrt(2*(((1-0.1)^(2/(1-0.1)))/(2*(1+0.1)) + (1/(2*(1-0.1))*(\x*\x) - (1/(1+(0.1)))*((\x*\x)^(0.5))^(1+0.1)))});     
      \draw[domain={(1-0.1)^(1/(1-0.1))+0.03:3},smooth,variable=\x,red] plot ({\x},{-sqrt(2*(((1-0.1)^(2/(1-0.1)))/(2*(1+0.1)) + (1/(2*(1-0.1))*(\x*\x) - (1/(1+(0.1)))*((\x*\x)^(0.5))^(1+0.1)))});
      \draw[domain={-3:-(1-0.1)^(1/(1-0.1))-0.03},smooth,variable=\x,red] plot ({\x},{-sqrt(2*(((1-0.1)^(2/(1-0.1)))/(2*(1+0.1)) + (1/(2*(1-0.1))*(\x*\x) - (1/(1+(0.1)))*((\x*\x)^(0.5))^(1+0.1)))});
      \draw[domain={-3:-(1-0.1)^(1/(1-0.1))-0.03},smooth,variable=\x,red] plot ({\x},{sqrt(2*(((1-0.1)^(2/(1-0.1)))/(2*(1+0.1)) + (1/(2*(1-0.1))*(\x*\x) - (1/(1+(0.1)))*((\x*\x)^(0.5))^(1+0.1)))});

      \draw[domain={1.3:3},smooth,variable=\x,black] plot ({\x},{sqrt(2*( (-1/(2*(1-0.1))*(1.3*1.3) + (1/(1+(0.1)))*((1.3*1.3)^(0.5))^(1+0.1))  + (1/(2*(1-0.1))*(\x*\x) - (1/(1+(0.1)))*((\x*\x)^(0.5))^(1+0.1)))});
      \draw[domain={1.3:3},smooth,variable=\x,black] plot ({\x},{-sqrt(2*( (-1/(2*(1-0.1))*(1.3*1.3) + (1/(1+(0.1)))*((1.3*1.3)^(0.5))^(1+0.1))  + (1/(2*(1-0.1))*(\x*\x) - (1/(1+(0.1)))*((\x*\x)^(0.5))^(1+0.1)))});
      \draw[domain={1.7:3},smooth,variable=\x,black] plot ({\x},{sqrt(2*( (-1/(2*(1-0.1))*(1.7*1.7) + (1/(1+(0.1)))*((1.7*1.7)^(0.5))^(1+0.1))  + (1/(2*(1-0.1))*(\x*\x) - (1/(1+(0.1)))*((\x*\x)^(0.5))^(1+0.1)))});
      \draw[domain={1.7:3},smooth,variable=\x,black] plot ({\x},{-sqrt(2*( (-1/(2*(1-0.1))*(1.7*1.7) + (1/(1+(0.1)))*((1.7*1.7)^(0.5))^(1+0.1))  + (1/(2*(1-0.1))*(\x*\x) - (1/(1+(0.1)))*((\x*\x)^(0.5))^(1+0.1)))});

      \draw[domain={-3:-1.3},smooth,variable=\x,black] plot ({\x},{sqrt(2*( (-1/(2*(1-0.1))*(1.3*1.3) + (1/(1+(0.1)))*((1.3*1.3)^(0.5))^(1+0.1))  + (1/(2*(1-0.1))*(\x*\x) - (1/(1+(0.1)))*((\x*\x)^(0.5))^(1+0.1)))});
      \draw[domain={-3:-1.3},smooth,variable=\x,black] plot ({\x},{-sqrt(2*( (-1/(2*(1-0.1))*(1.3*1.3) + (1/(1+(0.1)))*((1.3*1.3)^(0.5))^(1+0.1))  + (1/(2*(1-0.1))*(\x*\x) - (1/(1+(0.1)))*((\x*\x)^(0.5))^(1+0.1)))});
      \draw[domain={-3:-1.7},smooth,variable=\x,black] plot ({\x},{sqrt(2*( (-1/(2*(1-0.1))*(1.7*1.7) + (1/(1+(0.1)))*((1.7*1.7)^(0.5))^(1+0.1))  + (1/(2*(1-0.1))*(\x*\x) - (1/(1+(0.1)))*((\x*\x)^(0.5))^(1+0.1)))});
      \draw[domain={-3:-1.7},smooth,variable=\x,black] plot ({\x},{-sqrt(2*( (-1/(2*(1-0.1))*(1.7*1.7) + (1/(1+(0.1)))*((1.7*1.7)^(0.5))^(1+0.1))  + (1/(2*(1-0.1))*(\x*\x) - (1/(1+(0.1)))*((\x*\x)^(0.5))^(1+0.1)))});

      \draw[domain={-0.6:0.6},smooth,variable=\x,black] plot ({\x},{-sqrt(2*( (-1/(2*(1-0.1))*(0.6*0.6) + (1/(1+(0.1)))*((0.6*0.6)^(0.5))^(1+0.1))  + (1/(2*(1-0.1))*(\x*\x) - (1/(1+(0.1)))*((\x*\x)^(0.5))^(1+0.1)))});
      \draw[domain={-0.6:0.6},smooth,variable=\x,black] plot ({\x},{sqrt(2*( (-1/(2*(1-0.1))*(0.6*0.6) + (1/(1+(0.1)))*((0.6*0.6)^(0.5))^(1+0.1))  + (1/(2*(1-0.1))*(\x*\x) - (1/(1+(0.1)))*((\x*\x)^(0.5))^(1+0.1)))});
      \draw[domain={-0.3:0.3},smooth,variable=\x,black] plot ({\x},{-sqrt(2*( (-1/(2*(1-0.1))*(0.3*0.3) + (1/(1+(0.1)))*((0.3*0.3)^(0.5))^(1+0.1))  + (1/(2*(1-0.1))*(\x*\x) - (1/(1+(0.1)))*((\x*\x)^(0.5))^(1+0.1)))});
      \draw[domain={-0.3:0.3},smooth,variable=\x,black] plot ({\x},{sqrt(2*( (-1/(2*(1-0.1))*(0.3*0.3) + (1/(1+(0.1)))*((0.3*0.3)^(0.5))^(1+0.1))  + (1/(2*(1-0.1))*(\x*\x) - (1/(1+(0.1)))*((\x*\x)^(0.5))^(1+0.1)))});

      	\fill (-{(1-0.1)^(1/(1-0.1))},0) circle (0.05);
      	\fill ({(1-0.1)^(1/(1-0.1))},0) circle (0.05);
      	\fill (0,0) circle (0.05);
		\node [below] at ({(1-0.1)^(1/(1-0.1))},-0.1) {$e_+$};
		\node [below] at (-{(1-0.1)^(1/(1-0.1))},-0.1) {$e_-$};            

    \end{tikzpicture}
\caption{
A qualitative sketch of the level curves of $V$ is depicted above. 
The equilibria for the dynamical system are located at $(0,0)$ and $e_{\pm}=(\pm (1-p)^{1/(1-p)},0)$.
The level curves $V=c^*(p)$, that intersect $e_\pm$ are depicted in red.
Level curves with $V=c>c^*(p)$ and $V=c<c^*(p)$ are depicted in blue and black respectively.
The region enclosed by the red curves that contains $(0,0)$ is denoted by $\mathcal{H}$.
}  
\end{figure}

We now consider the structure of the level curves of $V$ in $\field{R}^{2}$ defined by
\begin{equation} \label{s3e} V(w,w')=c, \end{equation}
for $-\infty < c < \infty$. 
It is straightforward to establish that the family of level curves of $V$ are qualitatively as depicted in Figure \ref{fig:1}, for $0<p<1$, with $\mathcal{H}$ representing the parts of the level curve connecting $(\pm (1-p)^{1/(1-p)} , 0)$ that enclose the origin.
We denote $c^*(p)$ to be   
\begin{equation} \label{s3g} V(\pm (1-p)^{1/(1-p)},0) = \frac{(1-p)^{2/(1-p)}}{2(1+p)} =c^*(p)>0 .\end{equation}
Inside $\mathcal{H}$, the level curves are simple closed curves concentric with the origin $(0,0)$, and $V$ is increasing from $V=0$ at the origin $(0,0)$, as each level curve is crossed, when moving out from (0,0) to the boundary curve $\mathcal{H}$, on which $V=c^*(p)$. 
Thus, inside $\mathcal{H}$, $V$ has a minimum at $(0,0)$ and is increasing on moving radially away from $(0,0)$ to the boundary $\mathcal{H}$. 
We will focus attention on the level curves of $V$ on and inside $\mathcal{H}$, which have $0 \leq c \leq c^*(p)$.
We denote the interior of the level curve $V(w,w')=c$ by $\Omega_{c}$, with the level curve $V(w,w')=c$ labelled as $\partial \Omega_{c}$, for $ 0 \leq c \leq c^*(p)$. 

Now let $\tilde{w}:[0,\epsilon]\to\mathbb{R}$ be a local solution to (P) (any $\epsilon >0$) and define $F:[0,\epsilon] \to \field{R}$ to be,
\begin{equation} \label{s3i} F(\eta) = V(\tilde{w}(\eta),\tilde{w}' (\eta)) \ \ \  \forall \eta \in [0,\epsilon]. \end{equation}
Then $F \in C^{1}((0,\epsilon])$, and via \eqref{s3i}, \eqref{s3a}-\eqref{s3d}, \eqref{s2f} and \eqref{s2g}, $F$ satisfies, 
\begin{align}
\nonumber F'(\eta) & = \nabla V(\tilde{w}(\eta),\tilde{w}'(\eta))\cdot(\tilde{w}'(\eta), \tilde{w}''(\eta)) \\
\nonumber & = \nabla V(\tilde{w}(\eta),\tilde{w}'(\eta)) \cdot Q(\tilde{w}(\eta), \tilde{w}'(\eta),\eta) \\
\label{s3j} & = -\left(\frac{(n - 1)}{\eta}+\frac{\eta}{2}\right)(\tilde{w}'(\eta))^2 \quad \forall \eta \in (0,\epsilon].
\end{align}
We can now establish the following {\emph{a priori}} bound on solutions to (P), namely

\begin{lem} \label{l32}
Let $\tilde{w}:[0,\epsilon_2]\to\mathbb{R}$ be a local solution to (P) (any $0 \leq \epsilon_{1} < \epsilon_{2}$) with $0 < \alpha < (1-p)^{1/(1-p)}$ and $c=V(\tilde{w}(\epsilon_{1}), \tilde{w}' (\epsilon_{1}))$.
Then, 
\begin{equation*} \label{l32a} (\tilde{w}(\eta),\tilde{w}'(\eta)) \in \Omega_{c} \ \ \  \forall \eta \in (\epsilon_{1},\epsilon_{2}] . \end{equation*}
\end{lem}

\begin{proof}
Let $\epsilon_{1} = 0$ and note that
\begin{equation} \label{l32b} 0<c=V(\alpha,0) \leq c^*(p). \end{equation}
Via \eqref{s2c'} and \eqref{s2g}, we have $\tilde{w}''(0)<0$. 
Moreover, it follows from \eqref{s3j} that, $F'(\eta)<0$ almost everywhere on $(0,\epsilon_2$) with respect to the Lebesgue measure, and hence
\begin{equation} \label{l32c} F(\eta)<F(0) \ \ \ \forall \eta \in (0,\epsilon_{2}]. \end{equation}
Therefore, via \eqref{l32c}, \eqref{l32b} and \eqref{s3i},
\begin{equation*} \label{l32d} V(\tilde{w}(\eta),\tilde{w}'(\eta))<c \ \ \ \forall \eta \in (0,\epsilon_{2}], \end{equation*}
as required. 
The result follows similarly on the interval $(\epsilon_{1}, \epsilon_{2}]$ with $0 < \epsilon_1 < \epsilon_2$.
\end{proof}
We now have:

\begin{lem} \label{t33}
For $0 < \alpha < (1-p)^{1/(1-p)}$, (P) has a local solution $\tilde{w}:[0,\epsilon]\to\mathbb{R}$ (any $\epsilon >0$).
Moreover, these local solutions satisfy $(\tilde{w}(\eta),\tilde{w}'(\eta)) \in \Omega_{c}$ for all $\eta \in (0,\epsilon]$ with $c=V(\alpha,0)$.
\end{lem}

\begin{proof}
If $\alpha=0$ or $\alpha=(1-p)^{1/(1-p)}$, then there exists an equilibrium solution to (P) on $[0,\infty)$.
Alternatively, by Theorem \ref{t31} there exists $\epsilon_{1}>0$ (dependent on $\alpha$) such that (P) has a solution on $[0,\epsilon_{1}]$. 
Moreover, via Lemma \ref{l32}, if $0<\alpha<(1-p)^{1/(1-p)}$, (P) is {\emph{a priori}} bounded on $[0,\epsilon]$ (any $\epsilon >0$). 
Without loss of generality, suppose that $\epsilon >\epsilon_{1}>0$. 
Since $Q$ given by \eqref{s3a} is bounded on the set
\begin{equation*} \label{t33a} \mathcal{X} \subset \field{R}^3 : \mathcal{X} = \left\{ (w,w',\eta): |w| \leq (1-p)^{1/(1-p)},\ |w'| \leq \sqrt{2c^*(p)},\ \epsilon_{1} \leq \eta \leq \epsilon \right\} \end{equation*}
we can apply the Cauchy-Peano Local Existence Theorem \cite[Chapter 1, Theorem 2.1]{codd} repeatedly with
\begin{equation*} \label{t33b} \delta = \left( \max_{(w,w',\eta) \in \mathcal{X}} |Q(w,w',\eta)|\right)^{-1}, \end{equation*}
to establish that there exists a solution to (P) restricted to $[0,\epsilon]$. 
Since $\epsilon >0$ is arbitrary, the result follows, as required.
\end{proof}

\begin{thm} \label{c34}
For $0 < \alpha < (1-p)^{1/(1-p)}$, (P) has a solution $\tilde{w}:[0,\infty ) \to \mathbb{R}$. 
Moreover, these solutions to (P) satisfy $(\tilde{w}(\eta),\tilde{w}'(\eta)) \in \Omega_{c}$ for all $\eta \in (0,\infty)$, with $c=V(\alpha,0)$.
\end{thm}

\begin{proof}
The result follows directly from Lemmas \ref{l32} and \ref{t33}, since $\epsilon>0$ in Lemma \ref{t33} is arbitrary.
\end{proof}

\subsection{Uniqueness}
To begin this subsection we consider (P) with $\alpha=0$.

\begin{remk} \label{r35}
Let $\tilde{w}:[0,\infty ) \to \mathbb{R}$ be any solution to (P) restricted to $(0,\epsilon]$ with $\alpha=0$. 
It follows from \eqref{s3d}, \eqref{s3i} and \eqref{s3j} that
\begin{equation*} \label{r35a} V(\tilde{w}(\eta),\tilde{w}'(\eta))=F(\eta) \leq F(0)=V(0,0)=0 \ \ \  \forall \eta \in (0,\epsilon]. \end{equation*}
Thus $(\tilde{w}(\eta),\tilde{w}'(\eta)) \in \mathcal{S}$ for all $\eta \in (0,\epsilon]$, with $\mathcal{S}$ defined as the connected subset of 
\begin{equation*} \label{r35b} \{ (w,w') \in \field{R}^2 : V(w,w') \leq 0 \}\end{equation*}
which contains $(0,0)$. 
Hence $\mathcal{S}=\{(0,0)\}$ and so $(\tilde{w}(\eta),\tilde{w}'(\eta))=(0,0)$ for all $\eta \in (0,\epsilon]$. 
We conclude that the unique solution to (P) with $\alpha=0$ is given by the equilibrium solution $\tilde{w}\equiv 0$.
\end{remk}

Before we can establish a uniqueness result for (P), we require bounds on solutions to (P) when the solution is in a neighbourhood of the plane $w=0$. 

\begin{prop} \label{p36}
Let $w:[0, \infty) \to \field{R}$ be a solution to (P) such that $(w(\bar{\eta}),w'(\bar{\eta}))= (0, \beta) \in \Omega_{c^*(p)}$ with $\beta >0$. 
Then, 
\begin{equation*} \label{p36c} (1-p)^{1/(1-p)} \geq w(\eta) \geq \frac{\beta}{2}(\eta-\bar{\eta}), \ \ \  \frac{\beta}{2} \leq w'(\eta) \leq \beta \ \ \  \forall [\bar{\eta}, \bar{\eta} + \eta_{\beta}] ,
\end{equation*}
with 
\begin{equation} \label{p36a} \eta_{\beta}= \min \left \{\left( \frac{8}{7} \right)^{1/(n-1)} \bar{\eta}, \sqrt{\bar{\eta}^2 -4 \log \left(\frac{6}{7} \right)}, \bar{\eta} -\frac{\beta}{4m_{H}}, \frac{(1-p)^{1/(1-p)}}{\beta} \right \}
\end{equation}
and $m_H$ given by \eqref{s2m}.  
\end{prop}

\begin{proof}
Let $w:[0, \infty) \to \field{R}$ be any solution (P) which satisfies $(w(\bar{\eta}), w'(\bar{\eta}))=(0,\beta)$. 
It follows from Lemma \ref{l23} and an integration of \eqref{l23e} that
\begin{equation} \label{p36d} w'(\eta) = \beta \left(\frac{\bar{\eta}}{\eta} \right)^{n-1} e^{\frac{\bar{\eta}^2 -\eta^2}{4}} + \frac{1}{\eta^{n-1} e^{\frac{\eta^2}{4}}} \int_{\bar{\eta}}^{\eta} H(w(s)) s^{n-1} e^{\frac{s^2}{4}} ds \ \ \  \forall \eta \in [\bar{\eta}, \infty).
\end{equation}
Since 
\begin{equation*} \label{p36e} \left ( \frac{\bar{\eta}}{\eta} \right)^{n-1} e^{\frac{\bar{\eta}^2 -\eta^2}{4}} > \frac{7}{8} \cdot \frac{6}{7} \ \ \  \forall \eta \in \left[\bar{\eta}, \bar{\eta} + \min \left \{ \left( \frac{8}{7} \right)^{\frac{1}{n-1}} \bar{\eta}, \sqrt{\bar{\eta}^2 -4 \log \left(\frac{6}{7} \right)} \right\}\right]
\end{equation*}
and
\begin{equation} \label{p36f} \left| \frac{1}{\eta^{n-1} e^{\frac{\eta^2}{4}}} \int_{\bar{\eta}}^{\eta} H(w(s)) s^{n-1} e^{\frac{s^2}{4}} ds \right| < \frac{\beta}{4} \ \ \  \forall \eta \in \left[ \bar{\eta}, \bar{\eta} -\frac{\beta}{4m_{H}} \right] ,
\end{equation} 
it follows from \eqref{p36d}-\eqref{p36f} that
\begin{equation} \label{p36g} w'(\eta) > \frac{\beta}{2} \ \ \  \forall \eta \in [\bar{\eta}, \bar{\eta} + \eta_{\beta}] \end{equation}
with $\eta_{\beta}$ given by \eqref{p36a}. 
An integration of \eqref{p36g} then gives 
\begin{equation} \label{p36h} w(\eta) > \frac{\beta}{2}(\eta-\bar{\eta}) \ \ \  \forall \eta \in \left[\bar{\eta}, \bar{\eta} + \eta_{\beta} \right]. \end{equation}
Since $(w(\eta), w'(\eta)) \in \Omega_{c^*(p)}$ for all $\eta \in [\bar{\eta}, \infty)$, it follows from \eqref{p36g}, \eqref{p36h} and \eqref{s2g} that 
\begin{equation*} \label{p36i} w''(\eta)<0 \ \ \  \forall \eta \in [\bar{\eta}, \bar{\eta} + \eta_{\beta}] ,\end{equation*}
and hence,
\begin{equation} \label{p36j} w'(\eta) \in \left[\frac{\beta}{2}, \beta \right] \ \ \  \forall [\bar{\eta}, \bar{\eta} + \eta_{\beta}]. \end{equation}
The result follows from \eqref{p36h} and \eqref{p36j}, as required.
\end{proof}
Note that an analogous bounds to those in Proposition \ref{p36} hold for $(0,\beta) \in \Omega_{c^*(p)}$ with $\beta <0$.
Additionally, note that the {\emph{a priori}} bounds in Proposition \ref{p36} and symmetry in (P) allow us to establish the following uniqueness result for (P). 
The proof is based on the uniqueness argument originating in \cite{ag} and a local uniqueness result in \cite{meyer}. 

\begin{prop} \label{p37}
For $0 \leq \alpha \leq (1-p)^{1/(1-p)}$, (P) has a unique solution on $[0,\eta^*]$ for any $\eta^*>0$. 
\end{prop}

\begin{proof}
If $\alpha =0$, then uniqueness of the solution to (P) on $[0,\eta^*]$ follows from Remark \ref{r35}.
Now consider $0 < \alpha \leq (1-p)^{1/(1-p)}$. 
Since $Q$, in \eqref{s3a}, is Lipschitz continuous on $(\field{R}^2  \setminus \mathcal{N}) \times (0,\infty)$, for any neighbourhood $\mathcal{N}$ of the plane $w=0$, it follows from Theorem \ref{t31} and \cite[Chapter 1, Theorem 2.3]{codd}, that the solution to (P) is unique on $[0, \tilde{\eta}]$ for any $\tilde{\eta}>0$ such that 
\begin{equation*} \label{p37a} 
w(\eta)>0 \ \ \ \forall \eta \in [0,\tilde{\eta}]. 
\end{equation*}
Hence if $\alpha = (1-p)^{1/(1-p)}$, then the equilibrium solution to (P) given by $w\equiv (1-p)^{1/(1-p)}$ is unique on $[0,\eta^*]$.
We now consider $0 < \alpha < (1-p)^{1/(1-p)}$.  
Recall that any non-constant solution to (P) must be two signed. 
Suppose that there exists two distinct solutions to (P), denoted by $w_{i}:[0,\infty) \to \field{R}$ ($i=1,2$), for which
\begin{equation} \label{p37b} 
w_{1}(\eta) = w_{2}(\eta) \ \ \  \forall \eta \in [0,\bar{\eta}] 
\end{equation}
with $0<\bar{\eta}< \eta^*$ and for all $\epsilon>0$,
\begin{equation} \label{p37c} 
w_{1}(\eta) \neq w_{2}(\eta)  
\end{equation}
for some $\eta \in (\bar{\eta}, \bar{\eta} + \epsilon]$. 
From \cite[Chapter 1, Theorem 2.3]{codd} and Remark \ref{r35} it follows that for $i=1,2$ 
\begin{equation} \label{p37d} 
w_{i}(\bar{\eta}) =0, \ \ \  w_{i}'(\bar{\eta}) \neq 0. 
\end{equation}
Thus, there exists $\beta \in \field{R} \setminus \{0\}$ such that $(w_{i}(\bar{\eta}), w_{i}'(\bar{\eta})) = (0,\beta) \in \Omega_{c^*(p)}$. 
Without loss of generality (due to symmetry), we suppose that $\beta >0$. 
Let $\eta_{\beta}$ be given by \eqref{p36a}; so that it follows from Proposition \ref{p36} and Theorem \ref{c34} that 
\begin{equation} \label{p37e} 
\frac{\beta}{2}(\eta -\bar{\eta}) \leq w_{i}(\eta) \leq (1-p)^{1/(1-p)}, \ \ \  \frac{\beta}{2} \leq w_{i}'(\eta) \leq \beta \ \ \  \forall \eta \in [\bar{\eta}, \bar{\eta} + \eta_{\beta}]. 
\end{equation}
It follows immediately from \eqref{p37e} that 
\begin{equation} \label{p37f} 
|w_{1}(\eta)-w_{2}(\eta)|\leq(1-p)^{1/(1-p)}, \quad |w_{1}'(\eta)-w_{2}'(\eta)|\leq \beta \ \ \  \forall \eta \in [\bar{\eta}, \bar{\eta}+ \eta_{\beta}]. 
\end{equation}
Note that for $(W,W') \in [0,(1-p)^{1/(1-p)}]\times[0,\beta]$, then 
\begin{equation} \label{p37g} 
W+W^p+W' < (2+\beta^{1-p})(W+W')^p, 
\end{equation}
since $0<p<1$. 
Now via \eqref{s2f} and \eqref{s2g} respectively, we have,
\begin{equation} \label{p37h} 
|w_{1}(\eta)-w_{2}(\eta)|\leq \int_{\bar{\eta}}^{\eta} |w_{1}'(s)-w_{2}'(s)| ds, 
\end{equation}
\begin{equation} \label{p37i} 
|w_{1}'(\eta)-w_{2}'(\eta)|\leq \int_{\bar{\eta}}^{\eta} \frac{1}{(1-p)}|w_{1}(s)-w_{2}(s)| + |w_{1}(s)-w_{2}(s)|^p + \left(\frac{(n-1)}{s} + \frac{s}{2} \right)|w_{1}'(s) - w_{2}'(s)| ds 
\end{equation}
for all $\eta \in [\bar{\eta}, \bar{\eta}+\eta_{\beta}]$. 
We next introduce $v:[\bar{\eta}, \bar{\eta}+\eta_{\beta}] \to \field{R}$, given by
\begin{equation} \label{p37j} 
v(\eta)=|w_{1}(\eta)-w_{2}(\eta)| + |w_{1}'(\eta)-w_{2}'(\eta)| \ \ \  \forall \eta \in [\bar{\eta}, \bar{\eta}+\eta_{\beta}]. 
\end{equation}
Therefore, via \eqref{p37f}-\eqref{p37j}, it follows that
\begin{align}
\nonumber v(\eta) \leq \int_{\bar{\eta}}^{\eta} &  \frac{1}{(1-p)}|w_{1}(s)-w_{2}(s)| + |w_{1}(s)-w_{2}(s)|^p + \left(\frac{(n-1)}{s} + \frac{s}{2} +1 \right)|w_{1}'(s) - w_{2}'(s)|  ds \\
\nonumber < \int_{\bar{\eta}}^{\eta} & \frac{1}{(1-p)} \left ( \frac{(n-1)}{\bar{\eta}} + \frac{(\bar{\eta}+\eta_{\beta})}{2}+1 \right) (  |w_{1}(s)-w_{2}(s)| + |w_{1}(s)-w_{2}(s)|^p \\
\nonumber & +|w_{1}'(s)-w_{2}'(s)|) ds \\
\label{p37k} \leq \int_{\bar{\eta}}^{\eta} & \frac{1}{(1-p)} \left ( \frac{(n-1)}{\bar{\eta}} + \frac{(\bar{\eta}+\eta_{\beta})}{2}+1 \right) (2+\beta^{1-p} )(v(s))^p ds
\end{align}
for all $\eta \in [\bar{\eta}, \bar{\eta}+\eta_{\beta}]$ with the final inequality due to \eqref{p37f} and \eqref{p37g}. 
Also, via Proposition \ref{p36} and \eqref{p36a}, $\eta_{\beta}$ is dependent on $p, n, \bar{\eta}$ and $\beta$ only, and hence, it follows from \eqref{p37k} that
\begin{equation} \label{p37l} 
v(\eta) \leq \int_{\bar{\eta}}^{\eta} K(p, n, \bar{\eta}, \beta)(v(s))^p ds 
\end{equation}
for all $\eta \in [\bar{\eta}, \bar{\eta}+\eta_{\beta}]$, with constant $K(p, n, \bar{\eta},\beta)$ given by,
\begin{equation} \label{p37m} K(p, n, \bar{\eta}, \beta)= \frac{1}{(1-p)} \left ( \frac{(n-1)}{\bar{\eta}} + \frac{(\bar{\eta}+\eta_{\beta})}{2}+1 \right) (2+\beta^{1-p} ) > 0 . 
\end{equation}
Now, we introduce the function $J :[\bar{\eta}, \bar{\eta}+\eta_{\beta}] \to [0,\infty)$ given by 
\begin{equation} \label{p37n} 
J(\eta) = \int_{\bar{\eta}}^{\eta} K(p, n, \bar{\eta}, \beta)(v(s))^p ds \ \ \ \forall \eta \in [\bar{\eta}, \bar{\eta}+\eta_{\beta}]. 
\end{equation}
It follows from \eqref{p37m} and \eqref{p37n} that $J$ is non-negative, non-decreasing and differentiable on $[\bar{\eta}, \bar{\eta}+\eta_{\beta}]$, and via \eqref{p37l}, satisfies
\begin{equation} \label{p37o} 
(J(s))' \leq K(p, n, \bar{\eta}, \beta)(J(s))^p \ \ \  \forall s \in [\bar{\eta}, \bar{\eta}+\eta_{\beta}]. 
\end{equation}
Upon integrating \eqref{p37o} from $\bar{\eta}$ to $\eta$, we obtain
\begin{equation} \label{p37p} 
J(\eta) \leq ((1-p)K(p, n, \bar{\eta}, \beta)(\eta-\bar{\eta}))^{1/(1-p)} \ \ \  \forall \eta \in [\bar{\eta}, \bar{\eta}+\eta_{\beta}]. 
\end{equation}
Therefore, via \eqref{p37p}, \eqref{p37n} and \eqref{p37l} we have
\begin{equation} \label{p37q} 
v(\eta) \leq \delta \ \ \  \forall \eta \in [\bar{\eta}, \bar{\eta}+\eta_{\delta}] 
\end{equation}
with $\delta >0$ chosen sufficiently small so that 
\begin{equation} \label{p37r} 
\eta_{\delta} = \frac{\delta^{1-p}}{(1-p)K(p, n, \bar{\eta}, \beta)}  <\eta_{\beta}. 
\end{equation}
Now, from Proposition \ref{p36}, we have 
\begin{equation} \label{p37s} 
\min\{w_{1}(\eta), w_{2}(\eta) \} \geq \frac{\beta}{2}(\eta - \bar{\eta}) \ \ \ \forall \eta \in [\bar{\eta}, \bar{\eta}+\eta_{\beta}]. 
\end{equation}
Moreover, it follows from \eqref{s3a}, \eqref{p37s} and the mean value theorem, that there exists a function $\theta : ( \bar{\eta}, \bar{\eta} + \eta_{\beta}] \to \field{R} $ such that $\theta(s) \geq \min\{ w_{1}(s), w_{2}(s)\}$ on $(\bar\eta , \bar\eta + \eta_\beta]$, and for which 
\begin{align}
\nonumber |Q_{2}(w_{1}(s), & w_{1}'(s), s)-Q_{2}(w_{2}(s), w_{2}'(s), s)| \\
\nonumber \leq & \frac{1}{(1-p)}|w_{1}(s)-w_{2}(s)| + |w_{1}(s)^p-w_{2}(s)^p| + \left(\frac{(n-1)}{s} + \frac{s}{2}  \right)|w_{1}'(s) - w_{2}'(s)| \\
\nonumber \leq & \frac{1}{(1-p)}|w_{1}(s)-w_{2}(s)| + p(\theta(s))^{p-1} |w_{1}(s)-w_{2}(s)| + \left(\frac{n}{\bar{\eta}} + \frac{(\bar{\eta}+\eta_{\delta})}{2}  \right)|w_{1}'(s) - w_{2}'(s)| \\
\nonumber \leq & \left (\frac{1}{(1-p)} + p \left( \frac{\beta}{2} (s-\bar{\eta}) \right)^{p-1} \right) |w_{1}(s)-w_{2}(s)| + \left(\frac{n}{\bar{\eta}} + \frac{(\bar{\eta}+\eta_{\delta})}{2}  \right)|w_{1}'(s) - w_{2}'(s)| \\
\label{p37t} \leq & \left (\frac{1}{(1-p)} + p \left( \frac{\beta}{2} (s-\bar{\eta}) \right)^{p-1} + \frac{n}{\bar{\eta}} + \frac{(\bar{\eta}+\eta_{\delta})}{2} \right) v(s)
\end{align}
for each $s \in (\bar{\eta}, \bar{\eta}+\eta_{\beta}]$. 
Now, via \eqref{s2f}, \eqref{s2g}, \eqref{s3a}, \eqref{p37l}, \eqref{p37r}, \eqref{p37t} and \eqref{p37q}, we have, 
\begin{align}
\nonumber v(\eta)  \leq & \int_{\bar{\eta}}^{\eta}  | Q_{1}(w_{1}(s),w_{1}'(s),s) - Q_{1}(w_{2}(s), w_{2}'(s), s)| \\
\nonumber & + | Q_{2}(w_{1}(s),w_{1}'(s),s) - Q_{2}(w_{2}(s), w_{2}'(s), s) | ds \\
\nonumber \leq &\int_{\bar{\eta}}^{\bar{\eta} + \eta_{\delta}} K(p, n, \bar{\eta}, \beta) (v(s))^p ds \\
\nonumber & + \int_{\bar{\eta} +\eta_{\delta}}^{\eta} \left ( 1 + \frac{1}{(1-p)} + p \left(  \frac{\beta}{2} (s-\bar{\eta}) \right)^{p-1} + \frac{n}{\bar{\eta}} + \frac{(\bar{\eta}+\eta_{\beta})}{2} \right) v(s) ds \\
\label{p37u} \leq &\frac{\delta}{(1-p)} + \int_{\bar{\eta} +\eta_{\delta}}^{\eta} \left ( 1 + \frac{1}{(1-p)} + p \left(  \frac{\beta}{2} (s-\bar{\eta}) \right)^{p-1} + \frac{n}{\bar{\eta}} + \frac{(\bar{\eta}+\eta_{\beta})}{2} \right) v(s) ds
\end{align}
for all $\eta \in [\bar{\eta} + \eta_{\delta}, \bar{\eta} + \eta_{\beta}]$. 
An application of Gronwall's Lemma \cite[Corollary 6.2]{aman} to \eqref{p37u}, gives
\begin{align}
\nonumber v(\eta)  \leq & \frac{\delta}{(1-p)} \exp{ \biggl(  \int_{\bar{\eta} + \eta_{\delta}}^{\eta} \biggl ( 1 + \frac{1}{(1-p)} + p \biggl(  \frac{\beta}{2} (s-\bar{\eta}) \biggr)^{p-1} + \frac{n}{\bar{\eta}} + \frac{(\bar{\eta}+\eta_{\beta})}{2} \biggr) ds \biggr) } \\
\nonumber = & \frac{\delta}{(1-p)} \exp{ \biggl( (\eta - (  \bar{\eta} + \eta_{\delta})) \biggl( 1 + \frac{1}{(1-p)} +  \frac{n}{\bar{\eta}} + \frac{(\bar{\eta}+\eta_{\beta})}{2} \biggr) + 
    \biggl(  \frac{\beta}{2} \biggr)^{p-1} ((\eta-\bar{\eta})^{p} - (\eta_{\delta})^p )  \biggr) } \\
\label{p37v} \leq & \frac{\delta}{(1-p)} \exp{ \biggl( \eta_{\beta} \biggl ( 1+ \frac{1}{(1-p)} + \frac{n}{\bar{\eta}} + \frac{(\bar{\eta}+ \eta_{\beta})}{2} \biggr) + \biggl( \frac{\beta}{2} \biggr)^{p-1} \eta_{\beta}^{p} \biggr) }
\end{align}
for all $\eta \in [\bar{\eta} + \eta_{\delta}, \bar{\eta} + \eta_{\beta}]$. 
Since $v$ is non-negative, it follows from \eqref{p37v} and \eqref{p37q}, upon letting $\delta \to 0$, that 
\begin{equation} \label{p37w} 
v(\eta)=0 \quad \forall \eta \in [\bar{\eta}, \bar{\eta} + \eta_{\beta}]. 
\end{equation}
Moreover, it follows from \eqref{p37w} and \eqref{p37j} that 
\begin{equation*} \label{p37x} 
w_{1}(\eta ) = w_{2}(\eta )  \quad \forall \eta \in [\bar{\eta}, \bar{\eta} + \eta_{\beta}], 
\end{equation*}
which contradicts the definition of $\bar{\eta}$ in \eqref{p37b}-\eqref{p37c}. 
Thus, the solution $w_1:[0,\infty ) \to \mathbb{R}$ to (P) with $0< \alpha < (1-p)^{1/(1-p)}$ is unique on $[0,\eta^*]$ for any $\eta^*>0$, as required.
\end{proof}

\subsection{Continuous Dependence}

In this subsection we establish continuous dependence of solutions $w:[0,\infty )\to \mathbb{R}$ to (P) with respect to initial data $\alpha \in [0, (1-p)^{1/(1-p)})$. 
To proceed we establish that all such solutions to (P) satisfy $(w,w') \to (0,0)$ as $\eta \to \infty$. 
The uniqueness result in Proposition \ref{p37} then yields a local continuous dependence result (on arbitrarily large intervals), and finally, limiting behaviour of solutions to (P) as $\eta\to\infty$ allows continuous dependence to be established on $[0,\infty )$. 
To begin, we have

\begin{lem} \label{l38}
Let $w:[0,\infty) \to \mathbb{R}$ be the solution to (P) with $0<\alpha<(1-p)^{1/(1-p)}$. 
Then, for some $\eta_\alpha>0$, 
\begin{equation*} \label{l38a} 
|w'(\eta)| \leq \frac{4M_H}{\eta} \ \ \ \forall \eta \in [\eta_\alpha , \infty ) 
\end{equation*}
\end{lem}

\begin{proof}
Via \eqref{l23c} and \eqref{s2m},
\begin{equation} \label{l38e} 
|w'(\eta)| \leq M_{H}\frac{1}{\eta^{n-1}e^{\frac{1}{4}\eta^2}} \int_{0}^{\eta} s^{n-1}e^{\frac{1}{4}s^2} ds \ \ \ \forall \eta \in [0,\infty). 
\end{equation}
Via an application of Watson's Lemma \cite[Proposition 2.1]{Miller} we see that
\begin{equation} \label{l38f}
\int_{0}^{\eta} e^{\frac{1}{4}s^2} s^{n-1} ds 
= \frac{e^{\frac{1}{4}\eta^2 } \eta^n}{2} \int_{0}^{1} e^{-\frac{1}{4}\eta^2 q} (1-q)^{\frac{n}{2}-1} dq 
\sim \frac{e^{\frac{1}{4}\eta^2}\eta^n}{2} \left(\frac{4}{\eta^2}\right) \text{ as } \eta \to \infty.  
\end{equation}
Substituting \eqref{l38f} into \eqref{l38e} establishes that for sufficiently large $\eta_\alpha >0$, $w'$ satisfies
\begin{equation*} \label{l38h} 
|w'(\eta)|<\frac{4M_{H}}{\eta} \ \ \ \forall \eta \in [\eta_\alpha , \infty ), 
\end{equation*}
as required. 
\end{proof}

Additionally, we have,
\begin{lem} \label{l39}
Let $w:[0,\infty) \to \mathbb{R}$ be the solution to (P) with $0<\alpha<(1-p)^{1/(1-p)}$. 
Then, $F:[0,\infty) \to \field{R}$, as given by \eqref{s3i}, converges to $F_{\infty} \in [0,F(0))$ as $\eta \to \infty$.
\end{lem}

\begin{proof}
Theorem \ref{c34} ensures that $(w(\eta),w'(\eta)) \in \Omega_{c}$ for all $\eta \in (0,\infty)$ with $c=V(\alpha,0)=F(0)$, and so, via \eqref{s3i} and \eqref{s3j}, $F$ is continuously differentiable, non-increasing and bounded below by $0$. 
Therefore there exists $F_{\infty} \in [0, F(0))$, such that $F(\eta) \to F_{\infty}$ as $\eta \to \infty$, as required.
\end{proof}

\begin{thm} \label{t310}
Let $w:[0,\infty) \to \mathbb{R}$ be the solution to (P) with $0<\alpha<(1-p)^{1/(1-p)}$. 
Then,
\begin{equation*} \label{t310a} 
(w(\eta),w'(\eta)) \to (0,0) \ \ \ \text{ as } \eta \to \infty. 
\end{equation*}
\end{thm}

\begin{proof}
Recall from Theorem \ref{c34} that,
\begin{equation} \label{t310b} (w(\eta),w'(\eta)) \in \Omega_{c^*(p)} \ \ \ \forall \eta \in (0,\infty), \end{equation}
and from Lemma \ref{l38} that
\begin{equation} \label{t310c} w'(\eta) \to 0 \ \ \ \text{ as } \ \ \ \eta \to \infty. \end{equation}
In addition, via Lemma \ref{l39}, 
\begin{equation} \label{t310d} V(w(\eta),w'(\eta)) \to F_{\infty} \ \ \ \textrm{as } \ \ \ \eta \to \infty \end{equation}
for some $F_{\infty} \in [0,c^*(p))$. 
It follows from \eqref{t310b}-\eqref{t310d} that 
\begin{equation} \label{t310e} |w(\eta)| \to w_{\infty} \ \ \ \text{ as } \ \ \ \eta \to \infty \end{equation}
with $w_{\infty}$ the unique non-negative root of $V(w,0)= F_{\infty}$ for $w \in [0,(1-p)^{1/(1-p)})$.
Without loss of generality we suppose that $(w(\eta),w'(\eta)) \to (w_{\infty},0)$ as $\eta \to \infty$.
However it follows from \eqref{l23c} that 
\begin{equation} \label{t310h} 
w'(\eta ) = \frac{1}{\eta^{n-1}e^{\frac{1}{4}\eta^2}} \int_{0}^{\eta} H(w(s))s^{n-1}e^{\frac{1}{4}s^2} ds \ \ \ \forall \eta \in (0,\infty) 
\end{equation}
and $H(w_{\infty}) \leq 0$.
Using \eqref{t310e}, if $H(x_{\infty})<0$ then an application of Watson's Lemma to \eqref{t310h} implies that
\begin{equation} \label{t310k} w'(\eta) \sim \frac{2H(w_{\infty})}{\eta} \ \ \ \text{ as } \eta \to \infty. \end{equation}
In addition, from \eqref{s2f}, we have 
\begin{equation} \label{t310l} w(\eta )=\alpha + \int_{0}^{\eta} w'(s)ds \ \ \ \forall \eta \in [0,\infty), \end{equation}
which implies, via \eqref{t310k}, that
\begin{equation} \label{t310m}  w(\eta) \sim 2H(w_{\infty})\log({\eta}) \ \ \  \text{ as } \eta \to \infty, \end{equation}
which contradicts \eqref{t310e}. 
We conclude that $H(w_\infty )\not< 0$ and so we must have $H(w_{\infty})=0$. 
Since $w_{\infty} \in [0,(1-p)^{1/(1-p)})$, $H(w_\infty)=0$ requires that $w_{\infty}=0$. 
It then follows from \eqref{t310c} and \eqref{t310e} that, $(w(\eta), w'(\eta)) \to (0,0)$ as $\eta \to \infty$, as required.
\end{proof}

To establish continuous dependence for (P), we split the argument into three parts;
a local result on $[0, \eta_{1}]$ for $\eta_{1}$ small, to address the singularity in \eqref{s2g} as $\eta \to 0^+$; 
a local result on $[0, \eta_{2}]$ for $\eta_{2}$ arbitrarily large, via a `uniqueness implies continuous dependence' argument; 
and on $[\eta_2, \infty)$ via asymptotic behaviour of solutions to (P) as $\eta\to\infty$. 
Firstly, we have,

\begin{lem} \label{l311}
Let $w_1:[0,\infty) \to \mathbb{R}$ be the solution to (P) with $0<\alpha_1<(1-p)^{1/(1-p)}$. 
Then, for any $\epsilon >0$ there exists $\delta>0$ such that if $|\alpha_{1} - \alpha_{2}| < \delta$, the solution to (P) with $0<\alpha_2<(1-p)^{1/(1-p)}$, denoted by $w_2:[0,\infty) \to \field{R}$ satisfies
\begin{equation*} \label{l311a} \max\{|w_{1}(\eta) - w_{2}(\eta)| , |w_{1}'(\eta) - w_{2}'(\eta)|\} < \epsilon \ \ \ \forall \eta \in [0,\eta_{1}] \end{equation*}
with $\eta_{1} = \sqrt{ \frac{\alpha_{1}}{2 |m_{H}|}}$ for $m_H$ given by \eqref{s2m}.
\end{lem}

\begin{proof}
Via \eqref{l23b}
\begin{equation} \label{l311c} w_{i}(\eta) = \alpha_{i} + \int_{0}^{\eta} \frac{1}{t^{n-1}e^{\frac{t^2}{4}}} \int_{0}^{t} H(w_{\alpha_{i}}(s)) s^{n-1} e^{\frac{s^2}{4}} dsdt, \end{equation}
for all $\eta \in [0, \infty)$ and $i=1,2$. 
Let $0<\delta <\frac{\alpha_{1}}{4}$, then since $|w_{i}(\eta)| < (1-p)^{1/(1-p)}$ for all $\eta \in [0,\infty)$, via \eqref{l311c}, we have 
\begin{equation} \label{l311d}
w_{i}(\eta) \geq \alpha_{i} + \int_{0}^{\eta} \frac{1}{t^{n-1} e^{\frac{t^2}{4}}} \int_{0}^{t} s^{n-1} e^{\frac{s^2}{4}} m_{H} dsdt
> \frac{3\alpha_{1}}{4} + m_{H} \frac{\eta^2}{2}
\geq \frac{\alpha_{1}}{2}
\end{equation}
for all $\eta \in [0, \eta_{1}]$, $i=1,2$.
Additionally, via \eqref{l311c}, we have,
\begin{equation*} \label{l311e} 
|w_{1}(\eta) - w_{2}(\eta)| \leq |\alpha_{1} - \alpha_{2}| + \int_{0}^{\eta} \frac{1}{t^{n-1} e^{\frac{t^2}{4}}} \int_{0}^{t} | H(w_{1}(s)) - H(w_{2}(s))| s^{n-1}e^{\frac{s^2}{4}} dsdt \ \ \ \forall \eta \in [0, \eta_{1}]. 
\end{equation*}
Since $H$ given by \eqref{s2l} is bounded and Lipschitz continuous, on $\left[ \frac{\alpha_{1}}{2} , (1-p)^{1/(1-p)} \right]$, we have 
\begin{equation} \label{l311f} |H(w_{1}(\eta)) - H(w_{2}(\eta))| \leq K_{\alpha} | w_{1}(\eta) - w_{2}(\eta)| \ \ \ \forall \eta \in [0, \eta_{1}], 
\end{equation}
with $K_{\alpha}$ a Lipschitz constant for $H$ on $\left[\frac{\alpha_{1}}{2} , (1-p)^{1/(1-p)} \right]$. 
It follows from \eqref{l311d}-\eqref{l311f} that 
\begin{align} 
\nonumber |w_{1}(\eta) - w_{2}(\eta)| & \leq |\alpha_{1} - \alpha_{2}| + \int_{0}^{\eta} \int_{0}^{t} K_{\alpha} |w_{1}(s) - w_{2}(s)| dsdt \\
\label{l311g} & \leq |\alpha_{1} - \alpha_{2}| +  \int_{0}^{\eta} K_{\alpha} \eta_{1} |w_{1}(s) - w_{2}(s) | ds.
\end{align}
An application of Gronwall's Lemma to \eqref{l311g} yields 
\begin{equation*} \label{l311h} 
|w_{1}(\eta) - w_{2}(\eta)| \leq |\alpha_{1} - \alpha_{2}| e^{  K_{\alpha} \eta_{1} \eta } \leq |\alpha_{1} - \alpha_{2}| e^{K_{\alpha} \eta_{1}^2} \ \ \  \forall \eta \in [0, \eta_{1}]. 
\end{equation*}
Therefore, provided that $0 < \delta < \min \left\{ \frac{\alpha_{1}}{4} , \epsilon e^{-K_{\alpha} \eta_{1}^2} \right\}$, it follows that
\[ |w_{1}(\eta) - w_{2}(\eta)| < \epsilon \quad \forall \eta \in [0,\eta_{1}], \]
as required.
\end{proof}

Secondly, we have, 

\begin{lem} \label{l312}
Let $w_1:[0,\infty) \to \mathbb{R}$ be the solution to (P) with $0<\alpha_1<(1-p)^{1/(1-p)}$. 
Then, for any $\epsilon >0$ and any $\eta_2>0$ there exists $\delta>0$ such that if $|\alpha_{1} - \alpha_{2}| < \delta$, the solution to (P) with $0<\alpha_2 <(1-p)^{1/(1-p)}$, denoted by $w_2:[0,\infty) \to \field{R}$ satisfies
\begin{equation*} \label{l312a} \max \{ |w_{1}(\eta) - w_{2}(\eta) |, |w_{1}'(\eta) - w_{2}'(\eta) | \} < \epsilon \ \ \ \forall \eta \in [0, \eta_{2}]. \end{equation*}
\end{lem}

\begin{proof}
Without loss of generality suppose that $\eta_{2} > \eta_{1}$, for $\eta_{1}$ given in Lemma \ref{l311}.
It follows from Proposition \ref{p37} that the [IVP] given by \eqref{l312b}-\eqref{l312f}:
\begin{align}
\label{l312b} & (w)' = w' , \\
\label{l312c} & (w')' = H(w) - \left( \frac{(n-1)}{\eta} + \frac{\eta}{2} \right) w' \ \ \ \forall \eta \in [\eta_{1}, \eta_{2}] , \\
\label{l312e} & (w(\eta_{1}), w'(\eta_{1})) = (w_{i}(\eta_{1}) , w_i'(\eta_{1})) , \\
\label{l312f} & (w,w') \in C^1 ([ \eta_{1} , \eta_{2}]) ,
\end{align}
have unique solutions, given by $\left. (w_{i}, w_i' )\right|_{[\eta_{1}, \eta_{2}]}$ for $i=1,2$. 
Therefore, via \cite[Theorem 4.3, p.59]{codd}, there exists $\delta_{1} > 0$ such that provided
\begin{equation} \label{l312g} \max \{ |w_{1}(\eta_{1}) - w_{2}(\eta_{1}) | , |w_{1}'(\eta_{1}) - w_{2}'(\eta_{1})| \} < \delta_{1} \end{equation}
then
\begin{equation} \label{l312h} 
\max \{ | w_{1}(\eta) - w_{2}(\eta) | , |w_{1}'(\eta) - w_{2}'(\eta)| \} < \epsilon \quad \forall \eta \in [\eta_{1}, \eta_{2}]. 
\end{equation}
Setting $\epsilon =\delta_1$ in Lemma \ref{l311}, it follows that there exists $\delta >0$ such that for all $\alpha_{2}$ that satisfy $|\alpha_{1} - \alpha_{2}| < \delta$, we have
\begin{equation} \label{l312i} 
\max \{ |w_{1}(\eta) - w_{2}(\eta) | , |w_{1}'(\eta) - w_{2}'(\eta)| \} < \delta_{1} = \epsilon \ \ \ \forall \eta \in [0, \eta_{1}]. 
\end{equation}
The result follows from \eqref{l312g}-\eqref{l312i}, as required.
\end{proof}

Thirdly, we have,

\begin{lem} \label{l313}
Let $w_1:[0,\infty) \to \mathbb{R}$ be the solution to (P) with $0<\alpha_1<(1-p)^{1/(1-p)}$. 
Then for any $\epsilon >0$, there exists $\delta >0$ such that if $| \alpha_{1} - \alpha_{2} | < \delta$, the solution to (P) with $0<\alpha_2 <(1-p)^{1/(1-p)}$, denoted by $w_2:[0,\infty) \to \field{R}$ satisfies
\begin{equation*} \label{l313a} 
\max \{ |w_{1}(\eta) - w_{2}(\eta) | , |w_{1}'(\eta) - w_{2}'(\eta)| \} < \epsilon \quad \forall \eta \in [0, \infty). 
\end{equation*}
\end{lem}

\begin{proof}
Set $\epsilon >0$.
To begin, consider the level curves of $V$ denoted by $\partial\Omega_c$ that are closed and concentric with $(0,0)$.
For $0 < c < c^*(p)$, define the positive real numbers $w_c$ and $w_c'$ via the rules $V(w_c,0)=c$ and $V(0,w_c')=c$ respectively. 
Then for $r_{c}=\sqrt{w_c^2+{w_c'}^2}$, we have $\Omega_c \subset B_{r_{c}}(0,0)$ with $B_r((w,w'))$ denoting the Euclidean ball in $\mathbb{R}^2$ of radius $r$ with centre at $(w,w')$. 
Observe that $r_{c}\to 0$ as $c\to 0$.

Now, for any $\epsilon_a>0$, via Theorem \ref{t310}, there exists $\eta_a>0$ such that
\begin{equation} \label{l313c}
(w_{1}(\eta ),w_{1}'(\eta )) \in B_{\epsilon_a}(0,0) \ \ \ \forall \eta\in [ \eta_a,\infty ) .
\end{equation} 
Additionally, via Lemma \ref{l312}, for any $\epsilon_b >0$, there exists $\delta >0$ such that for all $|\alpha_1-\alpha_2|<\delta$, we have
\begin{equation} \label{l313d}
\max \{ |w_{1}(\eta ) - w_{2}(\eta )|, |w_{1}'(\eta ) - w_{2}'(\eta )|  \} < \epsilon_b \ \ \ \forall \eta \in [0,\eta_a].
\end{equation}
Via \eqref{l313c} and \eqref{l313d}, it follows that 
\begin{equation} \label{l313e}
(w_{2}(\eta ) ,w_{2}'(\eta )) \in B_{2(\epsilon_a+\epsilon_b)}(0,0) \ \ \ \forall \eta \in [\eta_a,\infty).
\end{equation}
Since $\Omega_c$ are open and have centre $(0,0)$, we can select $\epsilon_a$ and $\epsilon_b$ sufficiently small so that for some $c(\epsilon )\in (0,c^*(p))$, we have 
\begin{equation} \label{l313f} 
B_{2(\epsilon_a+\epsilon_b)}(0,0) \subset \Omega_{c(\epsilon)} \subset B_\epsilon(0,0) .
\end{equation}
Thus, it follows from \eqref{l313d}-\eqref{l313f} that
\begin{equation*} 
\max \{ |w_{1}(\eta ) - w_{2}(\eta )|, |w_{1}'(\eta ) - w_{2}'(\eta )|  \} < \epsilon \ \ \ \forall \eta \in [0,\infty ),
\end{equation*}
as required.
\end{proof}

\subsection{Summary}
We now amalgamate the main results in \S \ref{gwp} into the following well-posedness result for (P).

\begin{thm} \label{t314}
Let $ 0 \leq \alpha_{1} < (1-p)^{1/(1-p)}$. 
Then (P) has a unique solution $w_1 : [0, \infty) \to \field{R}$ and for any $\epsilon >0$ there exists $\delta >0$ such that for all $ 0\leq \alpha_{2} < (1-p)^{1/(1-p)}$ such that $|\alpha_{1} - \alpha_{2}| < \delta$, there exists a unique solution to (P) with $0 \leq \alpha_2< (1-p)^{1/(1-p)}$ denoted by $w_2: [0, \infty) \to \field{R}$ and 
\[ \max \{ |w_{1}(\eta) - w_{2}(\eta)|, |w_{1}'(\eta) - w_{2}'(\eta)| \} < \epsilon \ \ \ \forall \eta \in [0, \infty) . \]
Moreover, $(w_{i}, w_{i}') \to (0,0)$ as $\eta \to \infty$.
\end{thm}

\begin{proof}
Existence and uniqueness are given by Theorem \ref{c34}, Remark \ref{r35} and Proposition \ref{p37}. 
Continuous dependence on the initial data is established in Lemma \ref{l313} for $0 < \alpha_{1} < (1-p)^{1/(1-p)}$ and for $\alpha_{1} = 0$ see Theorem \ref{c34} and Remark \ref{r35}. 
Theorem \ref{t310} establishes that solutions to (P) tend to $(0,0)$ as $\eta \to \infty$, as required.
\end{proof}

\section{Qualitative Properties of solutions to (P)} \label{qp}
In this section, we establish that solutions $w:[0,\infty)\to\mathbb{R}$ to (P) with $0 < \alpha < (1-p)^{1/(1-p)}$, tend to $0$ algebraically as $\eta \to \infty$.
Furthermore, we establish that these solutions oscillate as $\eta \to \infty$.

The algebraic decay bounds here are established for solutions to (P), that are analogous to those in \cite{meyer} (for (P) with $0<p<1$ and $n=1$) and obtained via a bootstrap argument that appeared in \cite{weis} (for (P) with $p>1$ and $n\in\mathbb{N}$).
We note here that if one uses these algebraic decay bounds directly with oscillation theory for second order ordinary differential equations (see, for example \cite{kurt} or \cite{prot}), it does not appear possible to establish that solutions to (P) oscillate. 
Consequently the approach used to establish oscillation of solutions to (P) in what follows, is largely independent of standard methods from oscillation theory for second order ordinary differential equations.
More specifically, instead of employing a comparison principle of Sturmian-type for zeros of solutions to second order ordinary differential equations, we use a specific comparison theorem for solutions to second order semi-linear parabolic partial differential inequalities on an unbounded domain, which appeared in \cite[Theorem 2.8]{ag}.

\subsection{Algebraic Decay Bounds for Solutions to (P) as $\eta \to \infty$}

To begin, we have

\begin{prop} \label{p4.5}
Let $w:[0,\infty) \to \mathbb{R}$ be a solution to (P) with $0<\alpha<(1-p)^{\frac{1}{1-p}}$. 
Suppose that for $\sigma \in [0,\frac{2}{(1-p)}]$, and for $c_{1} >0$, that
\begin{equation} \label{p4.5a}
|w(\eta)| \leq \frac{c_{1}}{\eta^{\sigma}} \quad \forall \eta \in (0,\infty).
\end{equation}
Then,  
\begin{equation} \label{p4.5b}
|w'(\eta)| \leq \frac{1}{\eta^{\sigma p+1}} \left( \frac{M_H}{2} \sup_{s \in (0, \infty)} \left( s^{\frac{2p}{1-p} +2} e^{-\frac{3}{16}s^2}\right) + c_1^p 2^{\sigma p+2} \right) \quad \forall \eta \in (0,\infty).
\end{equation}
\end{prop}

\begin{proof}
Observe via \eqref{s2l}, \eqref{p4.5a} and Theorem \ref{c34} that
\begin{equation} \label{p4.5c}
|H(w(\eta))|= \left| \frac{1}{(1-p)}w(\eta)-|w(\eta)|^{p-1}w(\eta)\right| \leq |w(\eta)|^p \leq \frac{c_{1}^{p}}{\eta^{\sigma p}} \quad \forall \eta \in (0,\infty).
\end{equation}
Thus, via \eqref{l23c} and \eqref{p4.5c}, we have
\begin{align}
\nonumber |w'(\eta)| &\leq \frac{1}{\eta^{n-1} e^{\frac{\eta^2}{4}}} \int_{0}^{\eta} |H(w(s))| s^{n-1}e^{\frac{s^2}{4}} ds \\
\nonumber & \leq \frac{1}{e^{\frac{\eta^{2}}{4}}} \left( \int_{0}^{\frac{\eta}{2}} |H(w(s))|e^{\frac{s^2}{4}} ds + c_{1}^{p} \int_{\frac{\eta}{2}}^{\eta} \frac{e^{\frac{s^2}{4}}}{s^{\sigma p}} ds \right) \\
\label{p4.5d}  & \leq  \frac{M_{H}\eta e^{-\frac{3}{16} \eta^2}}{2}  + \frac{c_1^p2^{\sigma p+2}}{\eta^{\sigma p+1}}
\end{align}
for all $ \eta \in (0, \infty)$. 
Observe that
\begin{align}
\nonumber \frac{M_{H} \eta e^{-\frac{3}{16}\eta^2}}{2} & \leq \frac{M_H}{2\eta^{\sigma p +1}} \left( \eta^{\sigma p +2} e^{-\frac{3}{16}\eta^2} \right) \\
\nonumber & \leq \frac{M_H}{2\eta^{\sigma p +1}} \sup_{s \in (0, \infty)} \left( s^{\sigma p +2} e^{-\frac{3}{16}s^2} \right) \\
\label{p4.5e} & \leq \frac{M_H}{2\eta^{\sigma p +1}}\sup_{s \in (0, \infty)} \left( s^{\frac{2p}{1-p} +2} e^{-\frac{3}{16}s^2} \right) 
\end{align}
for all $\eta \in (0, \infty)$ and $\sigma \in \left[0, \tfrac{2}{(1-p)} \right]$. 
Substituting \eqref{p4.5e} into \eqref{p4.5d} yields \eqref{p4.5b}, as required.
\end{proof}

A simple consequence of Proposition \ref{p4.5} and Theorem \ref{t314} is

\begin{prop} \label{p4.6}
Let $w:[0,\infty) \to \mathbb{R}$ be a solution to (P) with $0<\alpha <(1-p)^{\frac{1}{1-p}}$. 
Then
\begin{equation*}
|w'(\eta)| \leq \frac{1}{\eta} \left( \frac{M_H}{2} \sup_{s \in (0, \infty)} \left( s^{\frac{2p}{1-p} +2} e^{-\frac{3}{16}s^2}\right) + 4(1-p)^{p/(1-p)} \right) \quad \forall \eta \in (0,\infty).
\end{equation*}
\end{prop}

\begin{proof}
It follows from Theorem \ref{t314}, that $(w,w')\in\Omega_{c^*(p)}$ for all $\eta \in [0,\infty )$.
The conclusion then follows from Proposition \ref{p4.5} (with $\sigma=0$, $c_{1}= (1-p)^{1/(1-p)}$), as required.
\end{proof}

We now establish the aforementioned algebraic decay bounds for solutions to (P) as $\eta\to\infty$. 

\begin{thm} \label{t4.7}
Let $w:[0,\infty) \to \mathbb{R}$ be a solution to (P) with $0<\alpha<(1-p)^{1/(1-p)}$. 
Then, for any $\epsilon>0$, there exists $c_{1\epsilon}, c_{2\epsilon} >0$ such that 
\begin{equation} \label{t4.7a}
|w(\eta)| < \frac{c_{1\epsilon}}{\eta^{\frac{2}{(1-p)}-\epsilon}} \quad \forall \eta \in (0,\infty),
\end{equation}

\begin{equation} \label{t4.7b}
|w'(\eta)| < \frac{c_{2\epsilon}}{\eta^{\frac{(1+p)}{(1-p)}-\epsilon}} \quad \forall \eta \in (0,\infty).
\end{equation}
\end{thm}

\begin{proof}
Observe on multiplying \eqref{s2b} by $\frac{w(\eta)}{\eta}$, we have, 
\begin{align}
\nonumber \frac{1}{\eta} \left[ |w(\eta)|^{1+p} - \frac{(w(\eta))^2}{(1-p)}\right] &= -\frac{w(\eta)w''(\eta)}{\eta}+ \frac{(1-n)w(\eta)w'(\eta)}{\eta^2} - \frac{w(\eta)w'(\eta)}{2} \\
\label{t4.7c} & = - \left[ \frac{(w(\eta))^2}{4} + \frac{w(\eta)w'(\eta)}{\eta} \right]' + \frac{(w'(\eta))^2}{\eta} - \frac{nw(\eta)w'(\eta)}{\eta^2}
\end{align}
for all $\eta \in (0,\infty)$.
Via Theorem \ref{t310}, $w(\eta) \to 0$ as $\eta \to \infty$ and hence there exists $\eta^* >0$ such that   
\begin{equation} \label{t4.7d}
|w(\eta)| \leq \left( \frac{2p(1-p)}{(1+p)} \right) ^{\frac{1}{(1-p)}} \quad \forall \eta \in [\eta^*, \infty).
\end{equation}
Additionally, given $F: [0,\infty) \to \mathbb{R}$, defined as in \eqref{s3i}, i.e.  
\begin{equation} \label{t4.7e}
F(\eta)=V(w(\eta), w'(\eta)) = \frac{(w'(\eta))^2}{2} - \frac{(w(\eta))^2}{2(1-p)} + \frac{|w(\eta)|^{1+p}}{(1+p)} \quad \forall \eta \in [0, \infty) ,
\end{equation}
we can refine our choice of $\eta^*$ in \eqref{t4.7d} so that we also have,
\begin{equation} \label{t4.7f}
0 \leq F(\eta) \leq \left( \frac{4((c(p))^{\frac{2}{1+p}}}{C(p)}\right)^{ \frac{(1+p)}{(1-p)}} \quad \forall \eta \in [\eta^*, \infty),
\end{equation}
with
\begin{equation} \label{t4.7g}
c(p)= \frac{1}{(1+p)}-\frac{1}{2} \quad \mathrm{and} \quad C(p)= \frac{2(1+p)}{(1-p)} +1.
\end{equation}
Thus, it follows from \eqref{t4.7e}, \eqref{t4.7d} and \eqref{t4.7c} respectively that 
\begin{align}
\nonumber \frac{F(\eta)}{\eta} &= \frac{(w'(\eta))^2}{2\eta} + \frac{1}{\eta} \left[-\frac{(w(\eta))^2}{2(1-p)} + \frac{|w(\eta)|^{1+p}}{(1+p)} \right] \\
\nonumber &\leq \frac{(w'(\eta))^2}{2\eta} + \frac{1}{\eta} \left[-\frac{(w(\eta))^2}{(1-p)} + |w(\eta)|^{1+p} \right] \\
\label{t4.7h} &= \frac{3(w'(\eta))^2}{2\eta} - \left[\frac{(w(\eta))^2}{4} + \frac{w(\eta)w'(\eta)}{\eta} \right]' -\frac{nw(\eta)w'(\eta)}{\eta^2},
\end{align}
for all $\eta \in [\eta^*, \infty)$. 
Since $F(\eta) \geq 0$ for all $\eta \in [0, \infty)$, together with the decay bound in Proposition \ref{p4.6} and Theorem \ref{t310}, it follows that we may integrate inequality \eqref{t4.7h} from $\eta$ ($\geq \eta^* > 1$) to $l$, and then allow $l \to \infty$, to obtain, 
\begin{align}
\nonumber \int_{\eta}^{\infty} \frac{F(t)}{t} dt &\leq \frac{3}{2} \int_{\eta}^{\infty} \frac{(w'(t))^2}{t} dt +\frac{(w(\eta))^2}{4} + \frac{w(\eta)w'(\eta)}{\eta} -n\int_{\eta}^{\infty} \frac{w(t)w'(t)}{t^2} dt \\
\label{t4.7i} &\leq \frac{(w(\eta))^2}{4} + \frac{1}{\eta}(1+n)\sup_{t\geq\eta}|w(t)w'(t)|+ \frac{3}{2} \int_{\eta}^{\infty} \frac{(w'(t))^2}{t} dt 
\end{align}
for all $\eta \in [\eta^*, \infty)$. 
Also, since $|w(\eta)| \leq (1-p)^{1/(1-p)}$ for all $\eta \in [0,\infty )$, we have
\begin{equation} \label{t4.7j}
F(\eta) \geq |w(\eta)|^{1+p}c(p) \geq 0 \quad \forall \eta \in [\eta^*, \infty) .
\end{equation}
Substituting \eqref{t4.7j} into \eqref{t4.7i} then yields
\begin{equation} \label{t4.7k}
0 < \int_{\eta}^{\infty} \frac{F(t)}{t} dt \leq \frac{1}{4}\left(\frac{F(\eta)}{c(p)}\right)^{\frac{2}{(1+p)}} +\frac{1}{\eta}(1+n)\sup_{t \geq \eta}|w(t)w'(t)| + \frac{3}{2} \int_{\eta}^{\infty} \frac{(w'(t))^2}{t} dt
\end{equation}
for $\eta \in [\eta^*, \infty)$. 
Observe that the right hand side of \eqref{t4.7k} is uniformly bounded for $\eta \in [\eta^*, \infty)$ via Proposition \ref{p4.6}.

Now suppose that there exists $k>0$ and $\sigma \geq 0$ such that 
\begin{equation} \label{t4.7l}
F(\eta) \leq \frac{k}{\eta^\sigma} \quad \forall \eta \in [\eta^*, \infty) .
\end{equation}
Via \eqref{t4.7j}, it follows that there exists a constant $c_{1} >0$ such that
\begin{equation} \label{t4.7m}
|w(\eta)| \leq \frac{c_{1}}{\eta^{\frac{\sigma}{(1+p)}}} \quad \forall \eta \in [\eta^*, \infty) .
\end{equation}
Thus, via Proposition \ref{p4.5} and \eqref{t4.7m}, there exists a constant $c_{2} >0$ such that 
\begin{equation} \label{t4.7n}
|w'(\eta)| \leq \frac{c_{2}}{\eta^{\frac{\sigma p}{(1+p)}+1}} \quad \forall \eta \in [\eta^*, \infty) .
\end{equation}
Hence, it follows from \eqref{t4.7k}-\eqref{t4.7n} and \eqref{t4.7f} that there exist constants $c_{3},c_{4},c_{5} >0$ such that
\begin{equation} \label{t4.7o}
\int_{\eta}^{\infty} \frac{F(t)}{t} dt \leq \frac{1}{4} \left (\frac{F(\eta)}{c(p)} \right)^{\frac{2}{(1+p)}} + \frac{c_{3}}{\eta^{\sigma+2}} + \frac{c_{4}}{\eta^{\frac{2\sigma p}{(1+p)}+2}} \leq \frac{F(\eta)}{C(p)} +\frac{c_{5}}{\eta^{\frac{2\sigma p}{(1+p)}+2}}
\end{equation}
for all $\eta \in [\eta^*,\infty)$. 
Setting $G:[\eta^*, \infty) \to \mathbb{R}$ to be
\begin{equation*} \label{t4.7p}
G(\eta)= \int_{\eta}^{\infty} \frac{F(t)}{t} dt \quad \forall \eta \in [\eta^*, \infty)
\end{equation*}
it follows from \eqref{t4.7o} that $G$ satisfies,
\begin{equation} \label{t4.7q}
(t^{C(p)}G(t))' \leq c_{6}t^{C(p) -3 - \frac{2 \sigma p}{(1+p)}} \quad \forall t \in [\eta^*, \infty)
\end{equation}
for some constant $c_6>0$. 
Provided that 
\begin{equation} \label{t4.7r}
C(p) -2 - \frac{2 \sigma p}{(1+p)} > 0,
\end{equation}
integrating inequality \eqref{t4.7q} from $\eta^*$ to $\eta$, yields
\begin{equation*} \label{t4.7s}
\eta^{C(p)}G(\eta) \leq \frac{c_{6}}{\left(C(p) -2 -\frac{2 \sigma p}{(1+p)} \right)} \frac{1}{\eta^{2+\frac{2\sigma p}{(1+p)}-C(p)}} + G(\eta^*)\eta^{*^{C(p)}} \quad \forall \eta \in [\eta^*, \infty),
\end{equation*}
for some constant $c_6>0$, and thus,
\begin{equation} \label{t4.7t}
G(\eta) \leq \frac{c_{7}}{\eta^{\frac{2\sigma p}{(1+p)} +2}} + \frac{c_{8}}{\eta^{C(p)}} \quad \forall \eta \in [\eta^*, \infty) ,
\end{equation}
for some constants $c_7,c_8>0$.
Recalling from \eqref{s3j}, that $F(\eta)$ is non-increasing on $[\eta^*,\infty)$, we have
\begin{equation} \label{t4.7u}
G(\eta) \geq \int_{\eta}^{2\eta} \frac{F(t)}{t} dt \geq \frac{1}{2} F(2\eta) \quad \forall \eta \in [\eta^*,\infty).
\end{equation}
Thus, it follows from \eqref{t4.7u} and \eqref{t4.7t} that 
\begin{equation} \label{t4.7v}
F(\eta) \leq \frac{c_{9}}{\eta^{\frac{2\sigma p}{(1+p)} +2}} + \frac{c_{10}}{\eta^{C(p)}} \quad \forall \eta \in [2\eta^*, \infty)
\end{equation}
for some constants $c_9,c_{10}>0$.
We now define $\bar{\sigma} : \left[0, \tfrac{2(1+p)}{(1-p)} \right] \times (0,1) \to \left[0, \tfrac{2(1+p)}{(1-p)} \right] $ given by 
\begin{equation*} \label{t4.7w}
\bar{\sigma} (\sigma, p)= \min \left\{ \frac{2\sigma p}{(1+p)} + 2, C(p) \right\} \quad \forall p \in (0,1),\ \sigma \in \left[ 0, \frac{2(1+p)}{(1-p)} \right].
\end{equation*}
Now since \eqref{t4.7l} is satisfied for $\sigma =0$ and $k=F(0)$, it follows from \eqref{t4.7v} that there exists a sequence $\{\sigma_{m}\}_{m \in \mathbb{N}}$ such that
\begin{equation} \label{t4.7x}
\sigma_{1} =0, \quad \sigma_{m+1} = \bar{\sigma}(\sigma_{m}, p)
\end{equation}
and 
\begin{equation} \label{t4.7y}
F(\eta) \leq \frac{k_{m}}{\eta^{\sigma_{m}}} \quad \forall \eta \in [\eta_m^*,\infty),
\end{equation}
for some constants $k_m>0$ ($m\in\mathbb{N}$, provided that $C(p) -3 - \frac{2 \sigma_m p}{(1+p)} > -1$, recalling \eqref{t4.7r}) and $\eta_m^*>0$. 
We obtain from \eqref{t4.7x} that,
\begin{equation} \label{t4.7z}
\sigma_{m} = \frac{2(1+p)}{(1-p)} - \frac{4p}{(1-p)} \left(\frac{2p}{(1+p)} \right)^{m-2} \quad \forall m \in \mathbb{N}
\end{equation}
and hence $\sigma_{m}$ is increasing with
\begin{equation} \label{t4.7aa}
\sigma_{m} \to \frac{2(1+p)}{(1-p)} \quad \text{as} \quad m \to \infty .
\end{equation}
Since
\begin{equation*}
C(p) -3 - \frac{2 \sigma_{m} p}{1+p} \geq \frac{1}{1-p} ( 2(1+p) -2(1-p) -4p) =0 >-1
\end{equation*}
it follows that $\sigma_{m}$ given by \eqref{t4.7z} satisfies \eqref{t4.7r} with $\sigma = \sigma_m$, and hence, via \eqref{t4.7y}, given $\epsilon >0$, there exists a sufficiently large $M \in \mathbb{N}$ such that 
\begin{equation} \label{t4.7ab}
|F(\eta)| \leq \frac{k_{M}}{\eta^{\sigma_{M}}} \quad \forall \eta \in [\eta^*, \infty)
\end{equation}
with $\sigma_{M} > \frac{2(1+p)}{1-p} - \epsilon(1+p)$ and $\eta^*=\eta_M^*$. 
Thus, via \eqref{t4.7ab} and \eqref{t4.7j}, there exists a constant $c_{11}>0$ such that 
\begin{equation} \label{t4.7ac}
|w(\eta)| \leq \frac{c_{11}}{\eta^{\frac{2}{(1-p)}- \epsilon}} \quad \forall [\eta^*, \infty).
\end{equation}
Since $|w(\eta)|$ is bounded, it follows that \eqref{t4.7ac} holds on $(0, \infty)$ (with a new constant $c_{1\epsilon}$). 
The proof is then completed by applying Proposition \ref{p4.5} to \eqref{t4.7ac} to obtain the conclusion for $|w'(\eta)|$, as required.
\end{proof}

\subsection{Oscillation of solutions to (P)}
We now establish that solutions to (P) oscillate as $\eta \to \infty$.
The approach we consider here relies on the uniform lower bound of solutions to the following Cauchy problem for a second order semi-linear parabolic partial differential equation related to [CP], given by:
\begin{align}
\label{ota} & u_t - \bigtriangleup u = \max\{ u,0 \}^p \ \ \ \text{ on } D_T, \\ 
\label{otb} & u=u_0 \ \ \ \text{ on }\partial D_T, \\ 
\label{otc} & u\in C^{2,1}(D_T)\cap C(\bar{D}_T) \cap L^\infty (\bar{D}_T) , 
\end{align}
with $0<p<1$ fixed, and $u_0: \partial D_T \to\field{R}$ is continuous, bounded, non-negative and non-zero on a set of positive Lebesgue measure. 
We denote the Cauchy problem given by \eqref{ota}-\eqref{otc} as [CP]$_+$. 
Moreover, we remark that [CP]$_+$ has been investigated in detail in \cite{ag} and notably, global existence and uniqueness of solutions has been established. 
To establish oscillation of solutions to (P), we construct a sequence of functions $\{ u^{(m)}\}_{m\in\field{N}}$ converging to a solution of [CP]$_+$ as $m\to\infty$ and compare the terms in the sequence to a solution of [CP] in a suitable subset of $\bar{D}_T$.

To begin, fix $u_0:\partial D_T\to\field{R}$ as specified in [CP]$_+$, and consider the sequence of Cauchy problems, given by: 
\begin{align}
\label{otd} & u_t^{(m)} - \bigtriangleup u^{(m)} = f_m(u^m) \ \ \ \text{ on } D_T, \\
\label{ote} & u^{(m)} =u_0 \ \ \ \text{ on } \partial D_T, \\
\label{otf} & u^{(m)}\in C^{2,1}(D_T)\cap C(\bar{D}_T) \cap L^\infty (\bar{D}_T) ,
\end{align}
for $m\in\field{N}$ and $f_m:\field{R}\to\field{R}$ given by
\begin{equation} \label{otg} f_m(u) = \begin{cases} 0, & u \leq 0 \\ m^{1-p}u , & 0 \leq u \leq \tfrac{1}{m}  \\ u^p , & u \geq \tfrac{1}{m} ,\end{cases} \end{equation}
with $0<p<1$.
For fixed $m\in\field{N}$ we refer to the Cauchy problem given by \eqref{otd}-\eqref{otg} as [CP]$_+^m$.

\begin{lem} \label{l41}
For fixed $u_0:\partial D_T\to\field{R}$, there exists a unique solution to [CP]$_+$ which we denote as $u:\bar{D}_\infty\to\field{R}$.
Moreover, for each $m\in\field{N}$, [CP]$_+^m$ has a unique solution $u^{(m)}:\bar{D}_\infty\to\field{R}$ which satisfies
\begin{equation*} \label{l41a} 
0 \leq u^{(m)}(x,t) \leq u^{(m+1)}(x,t) \leq u(x,t) \leq ((1-p)t + ||u_0||_\infty^{1-p})^{1/(1-p)} \ \ \ \forall (x,t)\in\bar{D}_\infty ,\ m\in\field{N} .
\end{equation*}
Additionally, 
\begin{equation} \label{l41b} 
\lim_{m\to\infty}u^{(m)}(x,t) = u(x,t) \ \ \ \forall (x,t)\in\bar{D}_\infty .
\end{equation}
\end{lem} 

\begin{proof}
Existence of a solution to [CP]$_+$ follows from \cite[Theorem 1.11]{ag}, and uniqueness follows from \cite[Corollary 2.18]{ag}.
Existence and uniqueness of solutions to [CP]$_m$ follows from standard theory since [CP]$_+$ is {\emph{a priori}} bounded on $\bar{D}_T$ for each $T>0$ and $f_m$ is locally Lipschitz continuous (see for example \cite{LUS}).
Since $f_m$ are locally uniformly H\"older continuous for all $m\in\field{N}$, by following the argument used to establish \cite[Theorem 1.7]{ag} with the sequence defined by \eqref{otd}-\eqref{otg} above (instead of \cite[(1.8)$_n$]{ag}) demonstrates that \eqref{l41b} holds.
\end{proof}

Immediately from Lemma \ref{l41} we have,

\begin{cor} \label{c42}
Let $u^{(m)}:\bar{D}_\infty\to\field{R}$ be as in Lemma \ref{l41}. 
Then,
\[ \sup_{m\in\field{N}}u^{(m)}(x,t) > ((1-p)t)^{1/(1-p)} \ \ \ \forall (x,t)\in D_\infty . \]
\end{cor}

\begin{proof}
From \cite[Lemma 2.2]{ag} it follows that $u(x,t)>((1-p)t)^{1/(1-p)}$ for all $(x,t)\in D_\infty$. 
The result then follows from \eqref{l41b}.
\end{proof}  

From Lemma \ref{l41} and Corollary \ref{c42} we can establish that solutions to (P) with $0<\alpha <(1-p)^{1/(1-p)}$ have zeros in any neighbourhood of $\infty$. 

\begin{lem} \label{l43}
Let $w:[0,\infty )\to\field{R}$ be a solution to (P) with $0< \alpha < (1-p)^{1/(1-p)}$. 
Then, for any $\eta^*>0$, there exists $\eta \in [\eta^* , \infty )$ such that $w(\eta )=0$.
\end{lem}

\begin{proof}
Suppose that for some $\eta^*>0$ that $w(\eta )\not= 0$ for all $\eta \in [\eta^* , \infty )$.
Now, define $\overline{u}:\overline{\Omega^*}\times \left[ 0,\tfrac{1}{2}\right]\to\field{R}$ as 
\begin{equation} \label{t43a} 
\overline{u}(x,t)= \left|w\left(\frac{|x|}{\left( t+\tfrac{1}{2}\right)^{1/2}}\right)\right|\left( t+\tfrac{1}{2}\right)^{1/(1-p)} \ \ \ \forall (x,t)\in \overline{\Omega^*}\times \left[ 0,\tfrac{1}{2}\right] ,
\end{equation}
with $\Omega^*:=\mathbb{R}^n\setminus \overline{B_{\eta^*}(0)}$ and with $B_r(x)$ representing the Euclidean ball in $\mathbb{R}^n$ of radius $r$ centred at $x\in\mathbb{R}^n$.
It follows immediately from \eqref{t43a} and the supposition, that
\begin{align} 
\label{t43b} & \overline{u}_t - \bigtriangleup \overline{u} - f_m(\overline{u}) = \overline{u}^p - f_m(\overline{u}) \geq 0 \ \ \ \text{ on }\Omega^*\times \left[ 0,\tfrac{1}{2}\right] , \\
\label{t43c} & \overline{u} \geq g \ \ \ \text{ on } \partial B_{\eta^*}(0) \times \left[ 0,\tfrac{1}{2}\right] , \\
\label{t43d} & \overline{u} \geq 0 \ \ \ \text{ on }  \overline{\Omega^*} \times\{ 0 \} , \\
\label{t43e} & \overline{u} \in C^{2,1}\left( \Omega^*\times \left( 0,\tfrac{1}{2}\right]\right) \cap C\left( \overline{\Omega^*}\times \left[ 0,\tfrac{1}{2}\right] \right) \cap L^\infty \left( \overline{\Omega^*}\times \left[ 0,\tfrac{1}{2}\right]\right) , 
\end{align}
with constant $g>0$ given by
\begin{equation*} g:= \inf_{\eta\in[\eta^*,\sqrt{2}\eta^*]} |w(\eta)| \left( \frac{1}{2}\right)^{1/(1-p)} , \end{equation*} 
and $f_m$ given by \eqref{otg}.
Now, set $\underline{u}:\bar{D}_\infty\to\field{R}$ to be $\underline{u}:=u^{(m)}$, with $u^{(m)}$ as in Lemma \ref{l41} for some $m\in\field{N}$ and fixed $u_0:\partial D_T\to\field{R}$ given by
\begin{equation} \label{t43f} 
u_0(x,0) = \begin{cases} \frac{g}{2}e^{-1/(\eta^*-|x|)}, & |x| \leq \eta^* \\ 0, & |x| \geq \eta^* .\end{cases} 
\end{equation}
Since $0 \leq u_0 \leq \tfrac{g}{2}$ on $\partial D_T$, it follows immediately from Lemma \ref{l41} and \eqref{t43f}, that 
\begin{align}
\label{t43g} & \underline{u}_t - \bigtriangleup \underline{u} - f_m(\underline{u}) = 0 \leq 0 \ \ \ \text{ on } \Omega^*\times (0,T], \\
\label{t43h} & \underline{u} \leq g \ \ \ \text{ on } \partial B_{\eta^*}(0) \times [0,T] , \\
\label{t43i} & \underline{u} = 0 \ \ \ \text{ on }  \overline{\Omega^*} \times\{ 0 \} , \\
\label{t43j} & \underline{u} \in C^{2,1}( \Omega^*\times (0,T])\cap C(\overline{\Omega^*}\times [0,T]) \cap L^\infty (\overline{\Omega^*}\times [0,T]) , 
\end{align}
with 
\[ T=\min\left\{ \frac{1}{2},\ \frac{g^{1-p}}{(1-p)}\left(1-\left(\frac{1}{2}\right)^{1-p}\right)\right\} .\] 

Therefore, from \eqref{t43b}-\eqref{t43e} and \eqref{t43g}-\eqref{t43j} respectively, it follows that $\overline{u}$ and $\underline{u}$ can be taken to be a bounded regular supersolution and a bounded regular subsolution on the $\overline{\Omega}^*\times [0,T]$ in \cite[Theorem 4.4]{Meyer1} (since $f_m$ is locally Lipschitz continuous), and hence 
\begin{equation} \label{t43k} 
\underline{u}\leq \overline{u} \ \ \ \text{ on } \overline{\Omega}^*\times [0,T] .
\end{equation} 
Since $m\in\field{N}$ used to define $\overline{u}$ is arbitrary, via \eqref{t43a}, \eqref{t43k} and Corollary \ref{c42}, it follows that 
\begin{equation} \label{t43l} 
\left|w\left(\frac{|x|}{\left( t+\tfrac{1}{2}\right)^{1/2}}\right)\right|\left( t+\tfrac{1}{2}\right)^{1/(1-p)} \geq \sup_{m\in\field{N}} u_m(x,t) > ((1-p)t)^{1/(1-p)} \ \ \ \forall (x,t)\in \overline{\Omega^*}\times (0,T] .
\end{equation}
Inequality \eqref{t43l} implies that $w(\eta )\not\to 0$ as $\eta \to \infty$, which contradicts Theorem \ref{t310}.
Hence, for every $\eta^*>0$, there exists some $\eta \in [\eta^* , \infty )$ such that $w(\eta ) = 0$, as required.  
\end{proof}

To establish that the zeros of non-trivial solutions to (P) are isolated, we have

\begin{lem} \label{l44}
Let $w:[0,\infty )\to\field{R}$ be a solution to (P) with $0 \leq \alpha < (1-p)^{1/(1-p)}$.
Suppose that there exists $\eta^*>0$ such that $(w(\eta^*),w'(\eta^*))=(0,0)$.
Then, $w\equiv 0$ on $[0,\infty )$.
\end{lem}

\begin{proof}
Without loss of generality, suppose that $\eta^*\in (0,\infty )$ is the smallest value of $\eta$ at which $(w(\eta),w'(\eta))=(0,0)$.
Consider $F:[0,\infty )\to\field{R}$ as in \eqref{s3i}, i.e. 
\[ F(\eta ) = V(w(\eta ), w'(\eta ))\ \ \ \forall \eta \in [0,\infty ) . \]
It follows from the argument in Remark \ref{r35} that $w \equiv 0$ on $[\eta^*,\infty )$. 
Now, consider $\eta \in [0,\eta^*)$. 
Via \eqref{s3j}, $F\in C^1((0,\infty ))$ and satisfies, 
\[ F'(\eta ) = -\left(\frac{(n-1)}{\eta} + \frac{\eta}{2} \right)(w'(\eta ))^2 = 
-2\left(\frac{(n-1)}{\eta} + \frac{\eta}{2} \right)\left( F(\eta) + \frac{w(\eta)^2}{2(1-p)} - \frac{|w(\eta)|^{1+p}}{(1+p)} \right) . \]
Thus, 
\begin{equation} \label{l44a}
F'(\eta ) + 2\left(\frac{(n-1)}{\eta} + \frac{\eta}{2} \right)F(\eta) = -2\left(\frac{(n-1)}{\eta} + \frac{\eta}{2} \right) \left(  \frac{w(\eta)^2}{2(1-p)} - \frac{|w(\eta)|^{1+p}}{(1+p)} \right) .
\end{equation}
Since $(w(\eta^*),w'(\eta^*)) =( 0,0)$ and $w,w'\in C^1((0,\infty ))$ it follows from \eqref{l44a} that there exists $\eta_*\in (0,\eta^*)$ such that 
\[F'(\eta ) + 2\left(\frac{(n-1)}{\eta} + \frac{\eta}{2} \right)F(\eta) \geq 0 \ \ \ \forall \eta \in (\eta_*,\eta^* ] , \]
and so 
\begin{equation} \label{l44b} 
\left( \eta^{2(n-1)}e^{\frac{1}{2}\eta^2}F(\eta )\right)'\geq 0 \ \ \ \forall \eta \in (\eta_*,\eta^*] . 
\end{equation}
Since $F(\eta^*)=0$, an integration of \eqref{l44b} yields
\begin{equation} \label{l44c} F(\eta ) \leq 0 \ \ \ \forall \eta \in (\eta_*,\eta^*] . \end{equation}
Since $V \geq 0$ in a sufficiently small neighbourhood of $(0,0)$ it follows that \eqref{l44c} and our supposition that $w\equiv 0$ on $[\eta_*,\eta^*]$ which contradicts the definition of $\eta^*$.
Therefore, it follows that $\alpha = 0$ and via Remark \ref{r35}, $w\equiv 0$ on $[0,\infty )$, as required. 
\end{proof}

We conclude from Lemmas \ref{l43} and \ref{l44} that solutions to (P) with $0<\alpha <(1-p)^{1/(1-p)}$ do not have non-isolated zeros in $[0,\infty )$, but have infinitely many isolated zeros in $[\eta^*,\infty )$ for any $\eta^*\in [0,\infty )$ i.e. solutions to (P) with $0 < \alpha < (1-p)^{1/(1-p)}$ oscillate as $\eta \to \infty$.

\section{Conclusion} \label{conc}

By amalgamating the conclusions of Theorem \ref{t314}, Theorem \ref{t4.7}, Lemma \ref{l43} and Lemma \ref{l44} into a statement about [CP], we have established that [CP] has a 1-parameter family of spatially inhomogeneous radially symmetrical solutions $u_\alpha :\bar{D}_T\to\mathbb{R}$ (any $T>0$ and $0<\alpha < (1-p)^{1/(1-p)}$) that oscillate as $|x|\to\infty$ for $t\in (0,T]$ and for which, $||u_\alpha (\cdot , t) ||_q$ is bounded for each $t\in [0,T]$ for any $q>\tfrac{(1-p)n}{2}$.    

As a consequence of the theory developed in this paper, we state the following improvements to the theory concerning homoclinic connections in \cite{meyer} that can be established using analogous arguments to those given in this paper (for $(\alpha , \beta )\in \Omega_{c^*(p)}\setminus \{ (0,0) \}$):
the solution to problem \cite[(P)]{meyer} is unique;
the problem \cite[(P)]{meyer} is continuously dependent on its data; and
solutions to \cite[(P)]{meyer} oscillate as $\eta \to \pm \infty$. 
This addresses one outstanding query in the conclusion of \cite{meyer}.
However the conjectured decay estimate for solutions to (P) as $\eta \to \infty$ remains open. 

We highlight here that the novel approach to establish that solutions to (P) oscillate as $\eta\to\infty$ was motivated by an apparent lack of sufficient conditions on solutions to (P) to apply Sturmian oscillation theory. 
Specifically, the decay bounds established in Theorem \ref{t4.7}, when used in conjunction with Sturmian oscillation theory for second order linear ordinary differential equations (see, for example \cite[p.42-46]{prot} or \cite{kurt}) appear to be insufficient to establish the oscillatory properties of solutions to (P).
In this direction, we note that if one could establish that solutions to (P) decay sufficiently rapidly, for instance,
\begin{equation} \label{conc1} |w(\eta)| \leq \frac{(16 - \epsilon )^{1/(1-p)}}{\eta^{2/(1-p)}} \ \ \ \text{ as } \eta \to \infty , \end{equation}
for some $\epsilon >0$, then one could use the aforementioned oscillation theory to establish that solutions to (P) oscillate as $\eta \to \infty$.
We also note here that an attempt refine Theorem \ref{t4.7} to establish the decay bound in \eqref{conc1} was undertaken by explicitly retaining the constants $c_i$ in the proof of Theorem \ref{t4.7} and passing to the limit as $m\to\infty$, but this was unsuccessful.

Now that the oscillatory properties of solutions to (P) as $\eta\to\infty$ have been established, a decay estimate for solutions to (P), as motivated by the formal estimate in \cite{meyer}, can potentially be established, thus classifying the remaining important property of solutions to (P) for $0<\alpha < (1-p)^{1/(1-p)}$.

Finally, we highlight a fundamental issue that arises from the previous consideration of [CP]. 
Consider the Cauchy problem given by \eqref{i1}, \eqref{i3} and 
\begin{equation} \label{conc2} u_0 = w_\alpha (|x|) \ \ \ \forall (x,0) \in \partial D_T , \end{equation}
with $w_\alpha : [0,\infty )\to\mathbb{R}$ the solution to (P) with $0 < \alpha < (1-p)^{1/(1-p)}$.
Immediately we infer that the Cauchy problem given by \eqref{i1}, \eqref{conc2} and \eqref{i3} has a global solution $u:\bar{D}_\infty\to\mathbb{R}$, given by 
\begin{equation*} u(x,t) = w_\alpha \left( \frac{|x|}{(t+1)^{1/2}}\right)((1-p)(t+1))^{1/(1-p)} \ \ \ \forall (x,t)\in\bar{D}_\infty. \end{equation*} 
However, uniqueness (and consequently continuous dependence on initial data) of solutions to the Cauchy problem given by \eqref{i1}, \eqref{conc2} and \eqref{i3} is not trivially settled.
A method which determines whether or not uniqueness holds for the Cauchy problem given by \eqref{i1}, \eqref{conc2} and \eqref{i3} would be a useful addition to the methods available for well-posedness results for boundary value problems for nonlinear parabolic partial differential equations.

\section*{Acknowledgments}
The authors would like to thank Prof. D. J. Needham for his helpful comments in relation to the preparation of this manuscript.




\section*{References}

\begin{thebibliography}{ww}
\bibitem{ag} J. Aguirre and M. Escobedo, ``A Cauchy problem for $u_{t} -\Delta u= u^p$ with $0<p<1$. Asymptotic behaviour of solutions." Annales Facult\'e des Sciences de Toulouse Math\'ematiques, \textbf{8}, 2, (1986), 175-203, \href{https://doi.org/10.5802/afst.637}{doi 10.5802/afst.637}
\bibitem{aman} H. Amann, \textit{Ordinary differential equations: an introduction to nonlinear analysis}. (de Gruyter, Berlin, 1990) \href{https://doi.org/10.1515/9783110853698}{doi 10.1515/9783110853698}
\bibitem{ball} J. Ball, ``Remarks on blow-up and nonexistence theorems for nonlinear evolution equations." Quart. J. Math., \textbf{28}, 4, (1977), 473-486, \href{https://doi.org/10.1093/qmath/28.4.473}{doi 10.1093/qmath/28.4.473}
\bibitem{cazy} T. Cazenave, F. Dickstein and F. B. Weissler, ``Sign-changing stationary solutions and blowup for the nonlinear heat equation in a ball." Math. Ann., \textbf{344}, 2, (2009), 431-449,  \href{https://doi.org/10.1007/s00208-008-0312-6}{doi 10.1007/s00208-008-0312-6}
\bibitem{caz} T. Cazenave, F. Dickstein and F. B. Weissler, ``Multi-scale multi-profile global solutions of parabolic equations in $\mathbb{R}^N$." Discrete Contin. Dyn. Syst. Ser., \textbf{5}, 3, (2012), 449-472, \href{https://doi.org/10.3934/dcdss.2012.5.449}{doi 10.3934/dcdss.2012.5.449}
\bibitem{codd} E. A. Coddington and N. Levinson, \textit{Theory of ordinary differential equations}. (McGraw-Hill, 1955, London)
\bibitem{deng} K. Deng and H. A. Levine, ``The role of critical exponents in blow-up theorems: the sequel." J. Math. Anal. Appl., \textbf{243}, 1, (2000), 85-126, \href{https://doi.org/10.1006/jmaa.1999.6663}{doi 10.1006/jmaa.1999.6663}
\bibitem{doh} C. Dohmen and M. Hirose, ``Structure of positive radial solutions to the Haraux Weissler equation." Nonlinear Anal., \textbf{33}, 1, (1998), 51-69, \href{https://doi.org/10.1016/S0362-546X(97)00542-7}{doi 10.1016/S0362-546X(97)00542-7}
\bibitem{esc} M. Escobedo and O. Kavian, ``Variational problems related to self-similar solutions of the heat equation." Nonlinear Anal., \textbf{11}, 10, (1987), 1103-1133 \href{https://doi.org/10.1016/0362-546X(87)90001-0}{doi 10.1016/0362-546X(87)90001-0}
\bibitem{fuji} H. Fujita, ``On the blowing up of solutions of the Cauchy problem for $u_t=\Delta u+ u^{1+\alpha}$." J. Fac. Sci. Univ. Tokyo Sect. I, \textbf{13}, 2, (1966), 109-124
\bibitem{weis} A. Haraux and F. B. Weissler, ``Non-uniqueness for a Semilinear Initial Value Problem." Indiana Univ. Math. J., \textbf{31}, 2, (1982), 167-189, \href{https://doi.org/10.1512/iumj.1982.31.31016}{doi 10.1512/iumj.1982.31.31016}
\bibitem{hi} M. Hirose and E. Yanagida, ``Global Structure of Self-Similar Solutions in a Semilinear Parabolic Equation." J. Math. Anal. Appl., \textbf{244}, 2, (2000), 348-368, \href{https://doi.org/10.1006/jmaa.2000.6706}{doi 10.1006/jmaa.2000.6706}
\bibitem{king} A. C. King and D. J. Needham, ``On a singular initial-boundary-value problem for a reaction-diffusion equation arising from a simple model of isothermal chemical autocatalysis." Proc. R. Soc. Lond. A., \textbf{437}, 1901, (1992), 657-671, \href{https://doi.org/10.1098/rspa.1992.0085}{doi 10.1098/rspa.1992.0085}
\bibitem{kurt} K. Kreith, \textit{Oscillation Theory.} (Springer, Basel, 1973) \href{https://doi.org/10.1007/BFb0067537}{doi 10.1007/BFb0067537}
\bibitem{LUS} O. A. Lady\u{z}enskaya, V. A. Solonnikov and N. N. Ural'ceva, \textit{Linear and Quasi-linear equations of parabolic type.} (AMS, Rhode Island, 1988)
\bibitem{lev} H. A. Levine, ``The role of critical exponents in blowup theorems." SIAM Rev., \textbf{32}, 2, (1990), 262-288 \href{https://doi.org/10.1137/1032046}{doi 10.1137/1032046}
\bibitem{cabe} P. M. McCabe, J. A. Leach and D. J. Needham, ``A note on the non-existence of permanent form travelling wave solutions in a class of singular reaction-diffusion problems." Dyn. Syst., \textbf{17}, 2, (2002), 131-135, \href{https://doi.org/10.1080/14689360110116498}{doi 10.1080/14689360110116498}
\bibitem{Meyer1} J. C. Meyer and D. J. Needham. ``Extended weak maximum principles for parabolic partial differential inequalities on unbounded domains." Proc. R. Soc. Lond. A, \textbf{470}, 2167, (2014), \href{https://doi.org/10.1098/rspa.2014.0079}{doi 10.1098/rspa.2014.0079}
\bibitem{Meyer15} J. C. Meyer and D. J. Needham. ``Well-posedness and qualitative behaviour of a semi-linear parabolic Cauchy problem arising from a generic model for fractional-order autocatalysis." Proc. R. Soc. Lond. A, \textbf{471}, 2175, (2015), \href{https://doi.org/10.1098/rspa.2014.0632}{doi 10.1098/rspa.2014.0632}
\bibitem{mey2} J. C. Meyer and D. J. Needham, ``On a $L^{\infty}$ functional derivative estimate relating to the Cauchy problem for scalar semi-linear parabolic partial differential equations with general continuous nonlinearity." J. Differential Equations, \textbf{265}, 8, (2018), 3345-3362, \href{https://doi.org/10.1016/j.jde.2018.04.051}{doi 10.1016/j.jde.2018.04.051}
\bibitem{mey} J. C. Meyer and D. J. Needham, \textit{The Cauchy problem for non-Lipschitz semi-linear parabolic partial differential equations.} (CUP, Cambridge, 2015) \href{https://doi.org/10.1017/CBO9781316151037}{doi 10.1017/CBO9781316151037}
\bibitem{meyer} J. C. Meyer and D. J. Needham, ``The evolution to localized and front solutions in a non-Lipschitz reaction-diffusion Cauchy problem with trivial initial data." J. Differential Equations, \textbf{262}, 3, (2017), 1747-1776, \href{https://doi.org/10.1016/j.jde.2016.10.027}{doi 10.1016/j.jde.2016.10.027}
\bibitem{Miller} P. D. Miller, \textit{Applied Asymptotic Analysis} (AMS, Rhode Island, 2006)
\bibitem{miz} N. Mizoguchi and E. Yanagida, ``Critical exponents for the blow-up of solutions with sign changes in a semilinear parabolic equation." Math. Ann., \textbf{307}, 4, (1997), 663-675, \href{https://doi.org/10.1007/s002080050055}{doi 10.1007/s002080050055}
\bibitem{miz2} N. Mizoguchi and E. Yanagida, ``Critical exponents for the blowup of solutions with sign changes in a semilinear parabolic equation, II." J. Differential. Equations, \textbf{145}, 2, (1998), 295-331, \href{https://doi.org/10.1006/jdeq.1997.3387}{doi 10.1006/jdeq.1997.3387}
\bibitem{naito} Y. Naito, ``Asymptotically self-similar behaviour of global solutions for semilinear heat equations with algebraically decaying initial data". Proc. Roy. Soc. Edinburgh: Sec. A Math., (2019), 1-23, \href{https://doi.org/10.1017/prm.2018.97}{doi 10.1017/prm.2018.97}
\bibitem{need} D. J. Needham, ``On the global existence of solutions to a singular semilinear parabolic equation arising from the study of autocatalytic chemical kinetics." ZAMP, \textbf{43}, 3, (1992), 471-480, \href{https://doi.org/10.1007/BF00946241}{doi 10.1007/BF00946241}
\bibitem{cham} D. J. Needham and P. G. Chamberlain, ``Global similarity solutions to a class of semilinear parabolic equations: existence, bifurcations and asymptotics." Proc. R. Soc. Lond. A, \textbf{454}, 1975, (1998), 1933-1959, \href{https://doi.org/10.1098/rspa.1998.0242}{doi 10.1098/rspa.1998.0242}
\bibitem{pol} P. Pol\'a\v{c}ik and E. Yanagida, ``On bounded and unbounded global solutions of a supercritical semilinear heat equation." Math. Ann., \textbf{327}, 4, (2003), 745-771, \href{https://doi.org/10.1007/s00208-003-0469-y}{doi 10.1007/s00208-003-0469-y}
\bibitem{prot} M. H. Protter and H. F. Weinberger, \textit{Maximum Principles in Differential Equations.} (Springer-Verlag, New York, 1984), \href{https://doi.org/10.1007/978-1-4612-5282-5}{doi 10.1007/978-1-4612-5282-5}
\bibitem{shi} N. Shioji and K. Watanabe, ``A generalised Poho\v{z}aev identity and uniqueness of positive radial solutions of $ \Delta u + g(r)u +h(r)u^p=0.$" J. Differential Equations, \textbf{255}, 12, (2013), 4448-4475, \href{https://doi.org/10.1016/j.jde.2013.08.017}{doi 10.1016/j.jde.2013.08.017}
\bibitem{fb} F. B. Weissler, ``Existence and non-existence of global solutions for a semilinear heat equation." Israel J. Math., \textbf{38}, 1-2, (1981), 29-40, \href{https://doi.org/10.1007/BF02761845}{doi 10.1007/BF02761845}
\bibitem{fbw} F. B. Weissler, ``Local existence and nonexistence for semilinear parabolic equations in $L^p$." Indiana Univ. Math. J., \textbf{29}, 1, (1980), 79-102, \href{https://doi.org/10.1512/iumj.1980.29.29007}{doi 10.1512/iumj.1980.29.29007}
\end{thebibliography}


\end{document}